\documentclass[9pt]{amsart}

\usepackage{amsmath}
\usepackage{amssymb}
\usepackage{amsthm}
\usepackage{enumerate}
\usepackage[all,arc,curve,color,frame]{xy}
\usepackage[utf8]{inputenc}

\usepackage{color, graphicx}
\usepackage[square,sort,comma,numbers]{natbib}

\renewcommand\section{\@startsection{section}{1}{\z@}
                                   {-3.5ex \@plus -1ex \@minus -.2ex}
                                   {2.3ex \@plus .2ex}
                                   {\normalfont\large\bfseries}}
\renewcommand\subsection{\@startsection{subsection}{2}{\z@}
                                   {-3.25ex\@plus -1ex \@minus -.2ex}
                                   {1.5ex \@plus .2ex}
                                   {\normalfont\normalsize\bfseries}}
\renewcommand\subsubsection{\@startsection{subsubsection}{3}{\z@}
                                   {-3.25ex\@plus -1ex \@minus -.2ex}
                                   {1.5ex \@plus .2ex}
                                   {\normalfont\normalsize\bfseries}}
\renewcommand\paragraph{\@startsection{paragraph}{4}{\z@}
                                   {3.25ex \@plus1ex \@minus.2ex}
                                   {-1em}
                                   {\normalfont\normalsize\bfseries}}

\makeatother

\newdimen\tableauside\tableauside=1.0ex
\newdimen\tableaurule\tableaurule=0.4pt
\newdimen\tableaustep
\def\phantomhrule#1{\hbox{\vbox to0pt{\hrule height\tableaurule
width#1\vss}}}
\def\phantomvrule#1{\vbox{\hbox to0pt{\vrule width\tableaurule
height#1\hss}}}
\def\sqr{\vbox{%
  \phantomhrule\tableaustep

\hbox{\phantomvrule\tableaustep\kern\tableaustep\phantomvrule\tableaustep}%
  \hbox{\vbox{\phantomhrule\tableauside}\kern-\tableaurule}}}
\def\squares#1{\hbox{\count0=#1\noindent\loop\sqr
  \advance\count0 by-1 \ifnum\count0>0\repeat}}
\def\tableau#1{\vcenter{\offinterlineskip
  \tableaustep=\tableauside\advance\tableaustep by-\tableaurule
  \kern\normallineskip\hbox
    {\kern\normallineskip\vbox
      {\gettableau#1 0 }%
     \kern\normallineskip\kern\tableaurule}%
  \kern\normallineskip\kern\tableaurule}}
\def\gettableau#1 {\ifnum#1=0\let\next=\null\else
  \squares{#1}\let\next=\gettableau\fi\next}

\makeatletter

\renewcommand\section{\@startsection{section}{1}{\z@}
                                   {-3.5ex \@plus -1ex \@minus -.2ex}
                                   {2.3ex \@plus .2ex}
                                   {\normalfont\large\bfseries}}
\renewcommand\subsection{\@startsection{subsection}{2}{\z@}
                                   {-3.25ex\@plus -1ex \@minus -.2ex}
                                   {1.5ex \@plus .2ex}
                                   {\normalfont\normalsize\bfseries}}
\renewcommand\subsubsection{\@startsection{subsubsection}{3}{\z@}
                                   {-3.25ex\@plus -1ex \@minus -.2ex}
                                   {1.5ex \@plus .2ex}
                                   {\normalfont\normalsize\bfseries}}
\renewcommand\paragraph{\@startsection{paragraph}{4}{\z@}
                                   {3.25ex \@plus1ex \@minus.2ex}
                                   {-1em}
                                   {\normalfont\normalsize\bfseries}}

\makeatother

\newcommand{\be}{\begin{equation}}
\newcommand{\ee}{\end{equation}}
\newcommand{\bea}{\begin{eqnarray}}
\newcommand{\eea}{\end{eqnarray}}
\newcommand{\ba}{\begin{array}}
\newcommand{\ea}{\end{array}}

\newcommand{\id}{\hbox{1\kern-.27em l}}

\newcommand{\half}{ {\textstyle \frac{1}{2}  } }

\newcommand{\al}{\alpha}

\newcommand{\bet}{\beta}

\newcommand{\la}{\lambda}

\newcommand{\tha}{\theta}

\newcommand{\cN}{\mathcal{N}}

\newcommand{\D}{{\rm d}}

\newcommand{\emp}{\emptyset}
\newcommand{\non}{\nonumber}

\newcommand{\SO}{\mathrm{SO}}

\newcommand{\Sp}{\mathrm{Sp}}
\newcommand{\su}{\mathrm{su}}

\newcommand{\Spin}{\mathrm{Spin}}

\newtheorem{lem}{Lemma}[section]
\newtheorem{prop}{Proposition}[section]
\newtheorem{thm}{Theorem}[section]
\newtheorem{remark}{Remark}[section]

\newtheorem{conjecture}{Conjecture}[section]

\newtheorem{Def}{Definition}[section]
\newtheorem{ex}{Example}
\newenvironment{rmk}{\begin{remark} \em}{\end{remark}}

\begin{document}

\title[Invariants of rigid surface operators]
{Invariants of rigid surface operators}

\author[Chuanzhong Li]{Chuanzhong Li}
\address[Chuanzhong Li]{$^1$ College of Mathematics and Systems Science, Shandong University of Science and Technology, Qingdao,266590, P.R.China}
\email{lichuanzhong@sdust.edu.cn}

\author[Bao Shou]{Bao Shou}
\address[Bao Shou]{$^2$ Center  of Mathematical  Sciences,
Zhejiang University,
Hangzhou 310027, China}
\email{bsoul@zju.edu.cn}

\date{}

\maketitle

\begin{abstract}
Lusztig used the symbol invariant to describe the Springer correspondence for classical groups. Similarly, the fingerprint invariant can describe the Kazhdan-Lusztig map. Both invariants pertain to rigid semisimple operators labeled by pairs of partitions $(\lambda', \lambda'')$.
It is conjectured that the symbol invariant is equivalent to the fingerprint invariant for rigid surface operators. In this study, we provide a proof of this conjecture.

We classify the maps that preserve the fingerprint invariant and demonstrate that they also preserve the symbol invariant. Conversely, we classify the maps that preserve the symbol invariant and show that they also preserve the fingerprint invariant. The constructions of the symbol and fingerprint invariants in prior works are crucial to the proof.

Additionally, we found that one condition in the definition of the fingerprint invariant is redundant for rigid surface operators. In the appendix, we describe  an alternative strategy to prove the equivalence of these invariants.
\end{abstract}

\tableofcontents

\newpage
\section{Introduction}
Surface operators are two-dimensional defects in four-dimensional gauge theory, which are generalizations of  line operators such as Wilson and 't Hooft operators.
Gukov and Witten introduced surface operators in their study of the ramified case of the Geometric Langlands Program, which is realized through $\mathcal{N}=4$ super Yang-Mills theories \cite{GW06}.
$S$-duality for  surface operators is studied   in \cite{Wit07}\cite{Wy09}\cite{Montonen:1977}
$$S: \; (G, \tau) \rightarrow  ( G^{L}, - 1 / n_{\mathfrak{g}} \tau)$$
with $n_{\mathfrak{g}}$ be 2 for $F_4$,  3 for $G_2$, and 1 for other semisimple classical groups. And  $\tau $ is a complexified  gauge coupling constant \cite{GW06}.
 We will focus on surface operators in  theories with gauge groups $\SO(2n)$ and $\Sp(2n)$,  whose Langlands dual group are $\SO(2n)$ and  $\SO(2n+1)$, respectively. Other gauge groups are trivial or more complicated.

In \cite{GW08}, Gukov and Witten extended their previous study of surface operators \cite{GW06}.  They identified a subclass of surface operators called {\it 'rigid'} surface operators, which are  anticipated  to be related to each other under $S$-duality \cite{GW08}\cite{Wy09}.
There are two types of rigid surface operators:  unipotent and semisimple. The rigid semisimple surface operators  are characterized  by pairs of  partitions $(\lambda^{'}, \lambda^{''})$. The partition is given by
$$\lambda_1^{n_1}\lambda_{2}^{n_{2}}\cdots\lambda_m^{n_m},\quad\quad\lambda_i \in \mathbb{N}$$
with $\la_1\ge \la_2 \ge \cdots \ge \la_m$ as shown in Fig.(\ref{lm}).
\begin{figure}[!ht]
  \begin{center}
    \includegraphics[width=2.5in]{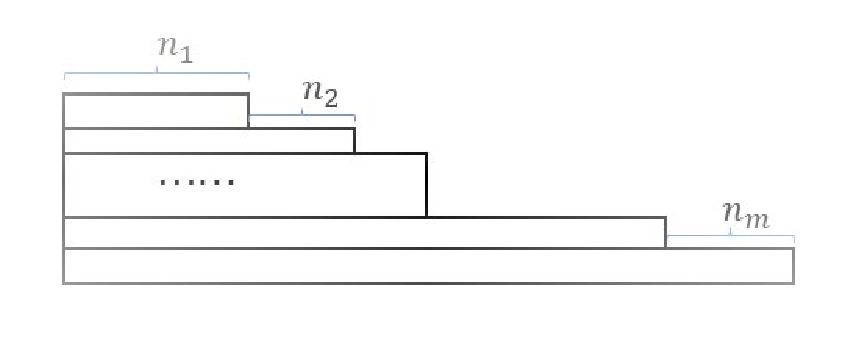}
  \end{center}
  \caption{ Partition  $\lambda_1^{n_1}\lambda_{2}^{n_{2}}\cdots\lambda_m^{n_m}$ with  length $l=\Sigma_i n_i$ .  }
  \label{lm}
\end{figure}
Unipotent rigid surface operators arise  when one partition of $(\lambda^{'}, \lambda^{''})$  is empty.

The invariant known as the {\it symbol} is based on the Springer correspondence \cite{CM93},  which is a map from the rigid semisimple surface operators to the unitary representation of the Weyl group.
Another invariant, termed the   {\it fingerprint } related to the Kazhdan-Lusztig map for the classical groups \cite{Lusztig:1984}\cite{Symbol 2}.
It   is a map from the rigid semisimple surface operators to the set of conjugacy classes of the Weyl group \cite{Spaltenstein:1992}. The rigid semisimple conjugacy classes(rigid surface operators) and  the conjugacy classes of the Weyl group are represented  by pairs of partitions $(\lambda^{'};\lambda^{''})$   and pairs of partitions  $[\alpha;\beta]$, respectively.
The symbol invariant  $\sigma$  also consists  of two partitions.
Thus the symbol invariant is a map from $(\lambda^{'};\lambda^{''})$  to $[\alpha;\beta]$.
Similarly, the fingerprint invariant  is  a map from  $(\lambda^{'};\lambda^{''})$ to  $\sigma$.

In \cite{Wy09}, Wyllard proposed  more explicit proposals for the action  of  $S$-duality maps  acting  on rigid surface operators,  utilizing  symbol invariant.   In \cite{Shou-sc}, a construction  of  the symbol invariant is provided.
This  construction makes the symbol invariant  to be calculated simply  and     convenient   to  find the $S$-duality maps of rigid surface operators.
The foundational properties of the fingerprint invariant and its construction are detailed in \cite{SW17-2}.
  A discrepancy in the total number of rigid surface operators between the $B_n$ and $C_n$
  theories, initially noted in \cite{GW08}\cite{Wy09}, was addressed using the symbol invariant. This approach allowed for the construction and classification of all problematic surface operators in the $B_n/C_n$
  theories \cite{rso}.

It has been verified that the symbol invariant of a rigid surface operator contains the same amount of information as the fingerprint invariant, as shown in Appendix B,  leading to the following conjecture \cite{Wy09}:
\begin{conjecture}
 The symbol invariant is equivalent to the  fingerprint invariant for rigid surface operators.
\end{conjecture}

In this study, we provide proof of this conjecture, as illustrated in Fig.(\ref{conjecture}).
\begin{figure}[!ht]
  \begin{center}
    \includegraphics[width=3in]{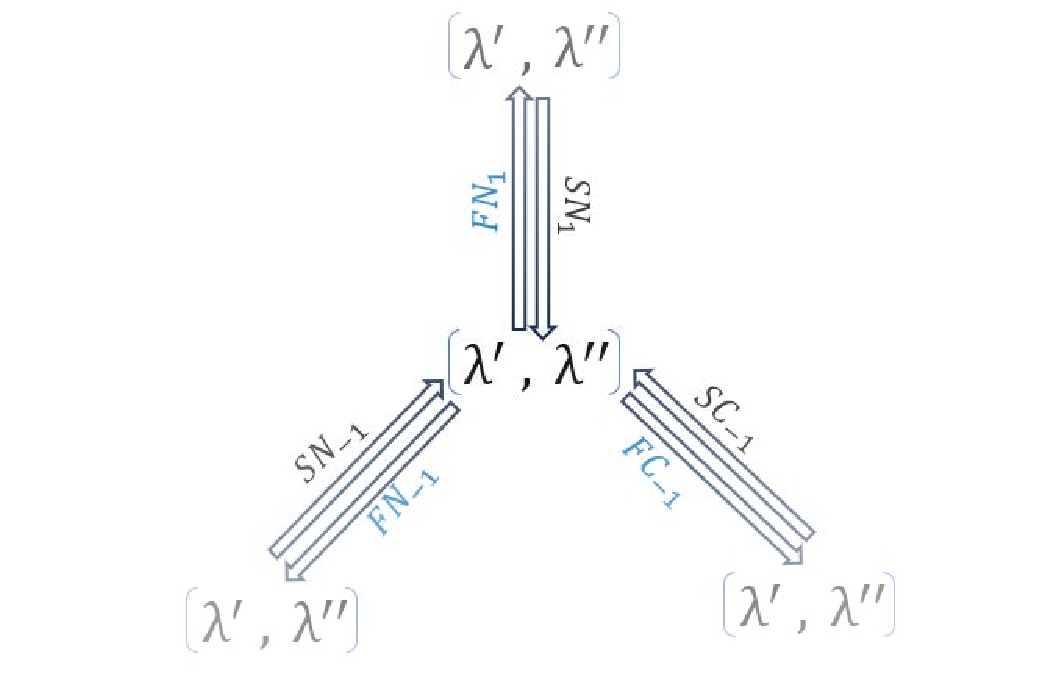}
  \end{center}
  \caption{Classification of invariants preserving maps: $FN_{1}$, $FN_{-1}$ and $FC_{-1}$ are fingerprint preserving maps.  $SN_{1}$, $SN_{-1}$ and $SC_{-1}$ are  symbol preserving maps. The four $(\lambda^{'};\lambda^{''})$ are different  rigid surface operators with the same invariants. }
  \label{conjecture}
\end{figure}
$FN_{1}$, $FN_{-1}$ and $FC_{-1}$ are the classification of the fingerprint preserving maps, we prove that they preserve the symbol invariant.
Here, \textit{map} (\textit{algorithm} or \textit{operation}) refers to a procedure that transforms one rigid surface operator into another.
Conversely,  $SN_{1}$, $SN_{-1}$ and $SC_{-1}$ are the classification of the symbol preserving maps, we demonstrate that they also preserve the fingerprint invariant.
The constructions of the symbol and the fingerprint in previous works play significant roles in proving the conjecture.

The following is an outline of  this article.  In Section \ref{ssymbol}, we  introduce  basic results related to rigid partitions and  the rigid  surface operators  as a  foundation.
In Section \ref{in}, we introduce  the definitions  of symbol invariant and fingerprint invariant in \cite{Wy09}\cite{Shou-sc}.
We also introduce the construction of symbol in \cite{Shou-sc} which is the basics of this study, as well as that of  the fingerprint invariant in \cite{SW17-2}.
In Section \ref{example},   we introduce examples of  maps preserving symbol and prove they preserve fingerprint, offering insights and hints for the proof of the conjecture.
We first prove the conjecture for $B_n$ theory and  restrict  the adjoint rows of the partition $\lambda=\lambda^{'}+\lambda^{''}$ to satisfy the \textit{rigid constraint}.
Under this constraint, the $C2$ condition in the definition of the fingerprint invariant is unnecessary.

In Section \ref{mapf}, we classify the fingerprint invariant preserving maps, which are related to the conditions in the definition of the fingerprint.
We then prove that these maps also preserve the symbol invariant.
Conversely, in Section \ref{maps}, we classify the symbol invariant preserving maps based on previous work on the construction of the symbol invariant \cite{Shou-sc}\cite{rso}. And then we prove they preserve the fingerprint  invariant.
In Section \ref{equ}, it is amazing to find that  the $C2$  condition  can be omitted for rigid surface operators. Thus, the complete proof of the conjecture for rigid surface operators in $B_n$
  theory can be directly obtained and generalized to $C_n$
  and $D_n$
  theories with minor modifications.
In Section \ref{summary},  we summarize the proof and mention potential applications in future research.

The paper includes two appendices.
Appendix A describes another strategy to prove the conjecture based on a different classification of symbol preserving maps.
Appendix B presents the rigid surface operators in the SO(11) and Sp(10) theories, which can be used to verify the proof.

\section{Partitions and Rigid Surface Operator}\label{ssymbol}
In this section, we provide the necessary background on surface operators as a foundation.   We closely  following  \cite{Wy09} to which we refer the reader for more details.

The  bosonic fields of $\mathcal{N}=4$ super-Yang-Mills theory: a gauge field as 1-form, $A_\mu$ ($\mu=0,1,2,3$), six real scalars, $\phi_I$ ($I=1,\ldots,6$), which  take values in the adjoint representation of the gauge group $G$.
Let the surface  $D$ be located  at  $x^0=0$  and $x^1=0$. Surface operators are supported on $D$ with  a certain singularity structure of  fields near the surface.
Since the fields satisfy the half BPS condition,
 the combinations $A=A_2\, \D x^2 +A_3\, \D x^3$ and $\phi = \phi_2 \,\D x^2 + \phi_3 \,\D x^3$ must satisfy  Hitchin's equations \cite{GW08}
\begin{equation}\label{hitch}
 F_A - \phi \wedge \phi = 0,\quad  \D_A\phi =0,\quad \D_A\star A = 0,
\end{equation}
which means a surface operator is defined as   a solution of Hitchin  equations with a prescribed singularity along the surface $D$.

Let $x_2+ix_3 = re^{i\tha}$.  Then the  most general possible rotation-invariant Ansatz for $A$ and $\phi$ is
\begin{eqnarray*}
A &=& a(r) \, \D \tha \,, \nonumber \\
\phi &=& -c(r) \, \D \tha + b(r) \frac{\D r }{r}  \,.
\end{eqnarray*}
Substituting this Ansatz into Hitchin's equations (\ref{hitch}) and defining $s = -\ln r$ ,  Hitchin  equations (\ref{hitch}) become Nahm's equations
\begin{eqnarray*} \label{nahm}
\frac{\D a}{\D s} &=& [b,c]\,, \nonumber \\
\frac{\D b}{\D s} &=& [c,a] \,,\\
\frac{\D c}{\D s} &=& [a,b]. \, \nonumber
\end{eqnarray*}
The surface operators, along   with   the communication relations for the  constants $a$, $b$ and $c$,  were discussed  in \cite{GW06}.
Another conformally invariant surface operator solution is given by \cite{GW08}:
\begin{equation*} \label{nahmshou}
a = \frac{t_x}{s + 1/f}\,,\qquad b = \frac{t_z}{s + 1/f}\,,\qquad c = \frac{t_y}{s + 1/f} \,
\end{equation*}
with $[t_x,t_y]=\varepsilon_{xyz}t_z $.
 $t_x,t_y$ and $t_z$  span a reducible representation  of   $\mathfrak{su}(2)$, which also belong to   the adjoint representation of the gauge group.

 Alternatively, the surface operators can be characterized  through the complexified conjugacy class of the monodromy
\begin{equation*}
U = P \exp(\oint \mathcal{A}) \,,
\end{equation*}
where $\mathcal{A} = A + i \phi$. The integration contour  is  a constant  circle near $r=0$.  From Hitchin's equations  (\ref{hitch}), it follows that   $\mathcal{F} = \D \mathcal{A} + \mathcal{A}\wedge \mathcal{A}=0$, which means that $U$ is independent of   deformations of the integration contour.
According to formula (\ref{nahmshou}), $U$ can be expressed as
\begin{equation} \label{Uplusshou}
U= P \exp(\frac{2\pi}{s+1/f} \,\,  t_+ ) \,,
\end{equation}
where $t_+\equiv t_x +i t_y$ is nilpotent,  corresponding to  unipotent surface operator.
When considering   a semisimple element $S$, the general monodromy of   a surface operator is
$$V=SU.$$
Both the unipotent and  semisimple classes    lead to  surface operators.
The field  $\Psi(r,\tha)$ associated  a   surface operator is required to satisfy  the following constraint  near the surface $D$ \cite{GW08}
\begin{equation} \label{Sun}
S \Psi(r,\tha) S^{-1} = \Psi(r,\tha+2\pi), \,
\end{equation}
which breaks  the gauge group to the centralizer of the seimisimple element $S$. Therefore   surface operators correspond to the classification of unipotent and semisimple conjugacy classes, which  have  been solved by mathematicians \cite{CM93}.

 Rigid surface operators are  closed on the $S$-duality and form  a subset of the surface operators constructed from conjugacy classes. 
 A unipotent conjugacy class is called rigid\footnote{The rigid surface operators  here  correspond to   strongly rigid operators in \cite{Wy09}. } if its dimension is strictly smaller than that of any nearby orbit.
 Similarly,  a semisimple conjugacy class $S$  is called rigid if the centralizer of such class is larger than that of any nearby class. In summary, rigid surface operators correspond to  monodromies of the form $V=SU$, where   $U$ is unipotent and rigid and $S$ is semisimple and rigid.
All rigid orbits have been classified \cite{GW08}.

   For  the rigid surface operators in  $B_n$($\SO(2n{+}1)$), $C_n$($\Sp(2n)$) and $D_n$($\SO(2n)$) theories,
the element $t_+$ in  Eq.(\ref{Uplusshou}) can be described in  block-diagonal basis  as follows
 \begin{equation}\label{ti}
 t_+ = \left( \begin{array}{ccc} t_+^{n_1}  & & \\
                        & \ddots &   \\
                        & & t_+^{n_l}
 \end{array} \right ),
 \end{equation}
where $t_+^{n_k}$ is the `raising' generator of the $n_k$-dimensional irreducible representation of $\su(2)$. There are restrictions on the allowed dimensions of the $\su(2)$ irreps.  The $t_i$'s  must should also belong to the adjoint representation of  the gauge group.
From the decomposition (\ref{ti}), unipotent (nilpotent) surface operators  are classified by the restricted partitions
 $$n_1\ge n_2 \ge \cdots \ge n_l$$
 with $\sum_{i=1}^l n_i = n$.
 Partitions correspond  one-to-one correspondence  with Young tableaux.  For instance,  the partition $3^22^31^3$  corresponds to
\be
\tableau{2 5 8}\non
\ee
Young tableaux find applications  in various branches  of mathematics and physics \cite{Localization}. They  are utilized in constructing   the eigenstates of the Hamiltonian System in  AGT correspondence \cite{Sh11}.

The classification of nilpotent orbits in terms of restricted partitions is as follows \cite{CM93}:
\begin{itemize}
  \item $(B_n)$: partitions of $2n+1$, $\sum \lambda_i=2n+1$, with all even part $\lambda_i$ appearing an even number of times.
  \item $(D_n)$: partitions of $2n$, $\sum \lambda_i=2n$, with all even part $\lambda_i$ appearing an even number of times.
  \item $(C_n)$: partitions of $2n$, $\sum \lambda_i=2n+1$, with all odd part $\lambda_i$ appearing an even number of times.
\end{itemize}
A partition  is called {\it rigid} if it satisfies the following conditions:
\begin{Def}\label{rigid}(\textbf{Rigid condition})
  \begin{enumerate}
  \item  No gaps i.e.~$\lambda_i-\lambda_{i+1}\leq1$ for all $i$.
  \item No odd (even) part appears exactly twice  in the $B_n$ or $D_n$($C_n$) theories.
\end{enumerate}
\end{Def}



We will focus on rigid partitions,  corresponding  to   rigid operators.
The following properties are crucial to this study, which are easily derivable from definitions \cite{Wy09}.
\begin{prop}{\label{Pb}}
For a rigid $B_n$ partition,
the longest row always contains an odd number of boxes. The two rows following the first row are either both of odd length or both of even length, continuing this pairwise pattern. If the Young tableau has an even number of rows, the shortest row must be of even length.
\end{prop}

\begin{prop}{\label{Pd}}
For a rigid $D_n$ partition, the longest row always contains an even number of boxes. The two rows following the first row are either both of even length or both of odd length, continuing this pairwise pattern. If the Young tableau has an even number of rows, the shortest row must be of even length.
\end{prop}

\begin{prop}{\label{Pc}}
For a rigid $C_n$ partition, the two longest rows both contain either an even or an odd number of boxes, continuing this pairwise pattern. If the Young tableau has an odd number of rows, the shortest row must contain an even number of boxes.
\end{prop}
\begin{flushleft}
Examples of partitions in the $B_n$,  $D_n$, and $C_n$ theories are shown in Fig.(\ref{bdcd}):
\end{flushleft}
\begin{figure}[!ht]
  \begin{center}
    \includegraphics[width=4.8in]{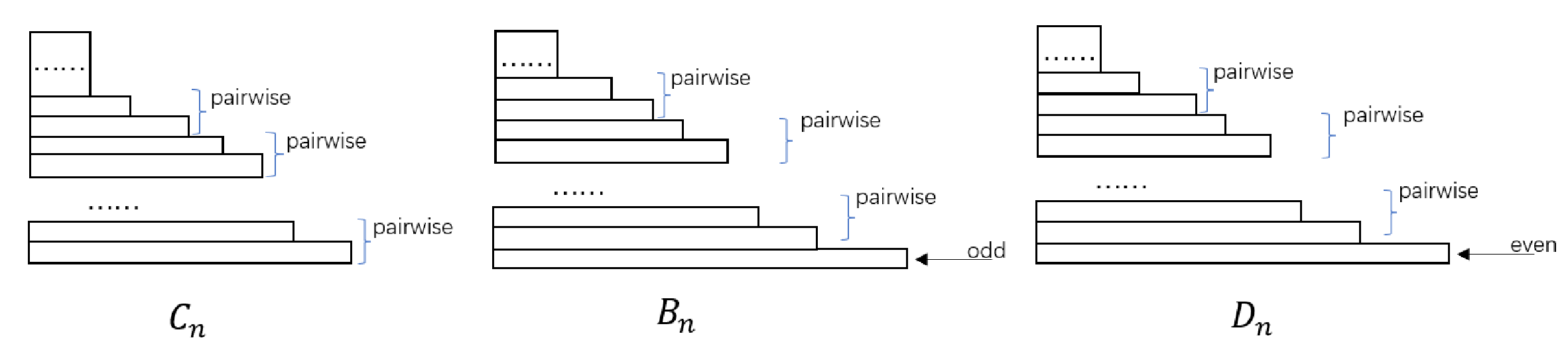}
  \end{center}
  \caption{ Partitions in the $C_n$,  $B_n$, and $D_n$  theories. }
  \label{bdcd}
\end{figure}

From Fig.(\ref{bdcd}), the following observations are frequently used in this study:
\begin{lem}{\label{lbcd}}
\begin{enumerate}
  \item There is an odd number of \textit{rows} under the bottom row of a  pairwise rows of  the partitions in the $B_n$ and $D_n$ theories.
  \item There is an even number of boxes above the top row of  a pairwise rows of  the partitions in the $B_n$ and $D_n$ theories.
  \item There is an even number of \textit{rows} under the bottom row of a  pairwise rows of  the partitions in the $C_n$ theory.
  \item There is   an  even  number of boxes above the top row of a  pairwise rows of  the partitions in the $C_n$ theory.
\end{enumerate}
\end{lem}

At the end of this section, the following conventions are introduced:
\begin{itemize}
  \item  For the rigid surface operator $(\lambda^{'}, \lambda^{''})$, the partition $\lambda$ is always given by
$$\lambda=\lambda^{'}+\lambda^{''},$$
where the  addition  of  partitions  is defined by the additions of each part
\begin{equation}\label{ap}
  \lambda_i+\mu_i.
\end{equation}
  \item $'Operator'$ or  $'rigid\,\, semisimple \,\,surface\,\, operator'$  will refer to the rigid  surface operator in the rest of the paper.
  \item For  the rigid surface operator $(\lambda^{'}, \lambda^{''})$ in $B_n$ theory, the first partition $\lambda^{'}$ is in $B_n$ theory and the second partition $\lambda^{''}$ is in $D_n$ theory.
  \item An even row  means the length of the row is even.
  \item  An even pairwise rows means a pair of even rows.
  \item Rows of a partition  are indexed from  bottom to top.
  \item Without causing confusion, a single row will be denoted with the same notation, such as $t$ and $b$, regardless of its length, the partition it belongs to, or its position in a pair of rows.
  \item Without causing confusion, the image of rigid  surface  operator  $(\lambda^{'}, \lambda^{''})$ will also be denoted as $(\lambda^{'}, \lambda^{''})$ under the symbol preserving such as the ones in Fig.(\ref{conjecture}).
\end{itemize}


\section{Invariants of Rigid Surface Operator}\label{in}
In this section, we introduce the definitions of the symbol invariant and fingerprint invariant of rigid surface operators, along with their respective constructions \cite{Shou-sc, Wy09, SW17-2}.

Invariants of rigid surface operators $(\lambda';\lambda'')$, such as the fingerprint invariant, symbol invariant, and dimension, remain unchanged under the $S$-duality map \cite{GW08, Wy09}.
The simplest invariant of rigid surface operators is the dimension $d$, given by the following formulas \cite{CM93, GW08}:
\begin{eqnarray*}
B_n: & d = 2n^2 + n -\half \sum_{k} (s_k')^2 -  \half \sum_{k} (s_k'')^2
+ \half \sum_{k\;\mathrm{odd}} r_k'+ \half \sum_{k\;\mathrm{odd}} r_k'' \,,\nonumber \\
D_n: & d =2n^2 - n -\half \sum_{k} (s_k')^2 -  \half \sum_{k} (s_k'')^2
+ \half \sum_{k\;\mathrm{odd}} r_k'+ \half \sum_{k\;\mathrm{odd}} r_k''\,, \\
C_n: & d = 2n^2 + n -\half \sum_{k} (s_k')^2 -  \half \sum_{k} (s_k'')^2
- \half \sum_{k\;\mathrm{odd}} r_k'- \half \sum_{k\;\mathrm{odd}} r_k''\,, \nonumber
\end{eqnarray*}
where $s'_k$ denotes the number of parts of $\lambda'$'s that are greater than or equal to $k$. And $r_k'$ denotes the number of parts of $\lambda'$ equal to $k$. Similarly,  $s_k''$ and $r_k''$   correspond  to the other partition $\lambda''$ of the rigid surface operator.

Other discrete quantum numbers, such as `centre' and `topology', are interchanged under $S$-duality \cite{GW08}.
The fingerprint invariant and the symbol invariant are finer than the dimension invariant $d$ and the discrete quantum numbers.
\subsection{Symbol Invariant of Partitions}
The symbol invariant is based on the Springer correspondence extended to rigid semisimple conjugacy classes $(\lambda';\lambda'')$
  \cite{Wy09}. In \cite{Shou-sc}, an equivalent definition of the symbol invariant for the $C_n$
  and $D_n$
  theories was presented, ensuring consistency with the $B_n$
  theory.
\begin{Def}\label{D2}\cite{Shou-sc}(\textbf{Symbol invariant})
 The  symbol  of a partition $\lambda$  is calculated  as follows.
\begin{itemize}
  \item  $B_n$ theory:
  \begin{enumerate}
    \item Add $l-k$  to the $k$th part of the partition.
    \item Arrange the odd parts and the even parts  of the sequence $l-k+\lambda_k$   in an increasing sequence $2f_i+1$ and in an increasing sequence $2g_i$, respectively.
    \item  Calculate the terms
 \begin{equation*}
   \al_i = f_i-i+1,\quad\quad \bet_i = g_i-i+1.
 \end{equation*}
    \item The {\it symbol} invariant  is written as follows
\begin{equation*}
  \sigma(\lambda)=\left(\ba{@{}c@{}c@{}c@{}c@{}c@{}c@{}c@{}} \al_1 &&\al_2&&\al_3&& \cdots \\ &\bet_1 && \bet_2 && \cdots  & \ea \right).
\end{equation*}
  \end{enumerate}
  \item  $C_n$ theory: There are two cases.
                               \begin{description}
                                \item[1]If the length of partition is even,  compute the symbol as that  in the $B_n$ case, and then append an extra $'0'$ on the left of the top row of the symbol.
                                \item[2]  If the length of the partition is odd, first append an extra $'0'$ as the last part of the partition. Then compute the symbol as that in the $B_n$ case. Finally,   delete a $'0'$ which is in the first entry of the bottom row of the symbol.
                              \end{description}
   \item  $D_n$ theory: first append an extra $'0'$ as the last part of the partition, and then compute the symbol as in the $B_n$ case. Finally,  delete two 0's which are  in the first two entries of the bottom row of the symbol.
\end{itemize}
\end{Def}
For the symbol $\sigma$ of  rigid semisimple  operators $(\lambda',\lambda'')$, it  is obtained by  adding  the entries that are `in the same place' of symbols of $\lambda'$ and $\lambda''$
\begin{equation}\label{ddddr}
  \sigma((\lambda^{'};\lambda^{''}))=\sigma(\lambda^{'})+\sigma(\lambda^{''}).
\end{equation}
For example,
\begin{equation*} \label{symboladd}
\left(\begin{array}{@{}c@{}c@{}c@{}c@{}c@{}c@{}c@{}c@{}c@{}c@{}c@{}c@{}c@{}} 0&&0&&0&&1&&1&&1&&2 \\ & 1 && 1 && 1 &&1&&1&&2 & \end{array} \right) +
 \left(\begin{array}{@{}c@{}c@{}c@{}c@{}c@{}c@{}c@{}c@{}c@{}c@{}c@{}} 0&&0&&0&&1&&1&&1 \\ & 1 && 1 &&1&&1&&1 & \end{array} \right)=
\left(\begin{array}{@{}c@{}c@{}c@{}c@{}c@{}c@{}c@{}c@{}c@{}c@{}c@{}c@{}c@{}} 0&&0&&0&&1&&2&&2&&3 \\ & 1 && 2 && 2 &&2&&2&&3 & \end{array} \right).
\end{equation*}


Table \ref{newt} summarizes the contribution to the symbols of rows in rigid surface operators concisely \cite{rso}.
\begin{prop}\label{row-eo}
 The contribution to the symbol of the bottom row of  an odd pairwise rows with length   $L$ is the   same as that of  the top row in an even pairwise rows  with length $L+1$.
 The contribution to the symbol of the top row of an odd  pairwise rows with length $L$  is the   same as that of  the bottom  row in an even  pairwise row with length $L-1$.
\end{prop}

Note that   Proposition \ref{row-eo} and Table \ref{newt} are \textit{independent} of the specific theories. This independence implies that the contribution to the symbol of a row is invariant.
In other words, given the contribution to the symbol, there are always a finite number of possible lengths and locations for a row.
\begin{rmk}\label{set33}
The following three sets give the same symbol invariant:
 \begin{enumerate}
    \item Symbol invariant $\sigma((\lambda^{'},\lambda^{''}))$.
   \item Rigid surface operator $(\lambda^{'},\lambda^{''})$.
   \item Length and location of each row in $(\lambda^{'},\lambda^{''})$.
 \end{enumerate}
\end{rmk}

\begin{table}
\begin{tabular}{|c|c|c|c|}\hline
Parity of length of row & Length & Location in a pairwise rows   & Contribution    \\ \hline
odd & $2n+1$ & top & $\Bigg(\!\!\!\ba{c}0 \;\; 0\cdots 0\;\; 0 \cdots 0 \\
\;\;\;0\cdots \underbrace{1 \;\;1\cdots 1}_{n} \ \ea \Bigg)$   \\ \hline
odd & $2n+1$& bottom   & $\Bigg(\!\!\!\ba{c}0 \;\; 0\cdots \overbrace{ 1\;\; 1\cdots1}^{n+1} \\
\;\;\;0\cdots 0\;\; 0\cdots 0 \ \ea \Bigg)$    \\ \hline
even & $2m$ & bottom   &  $\Bigg(\!\!\!\ba{c}0 \;\; 0\cdots 0\;\; 0 \cdots 0 \\
\;\;\;0\cdots \underbrace{1 \;\;1\cdots 1}_{m} \ \ea \Bigg)$  \\ \hline
even & $2m$&top    &   $\Bigg(\!\!\!\ba{c}0 \;\; 0\cdots \overbrace{ 1\;\; 1\cdots1}^{m} \\
\;\;\;0\cdots 0\;\; 0\cdots 0 \ \ea \Bigg)$      \\ \hline
\end{tabular}
\caption{ Contributions  of rows in partitions to  symbol in  the $B_n$, $D_n$, and $C_n$ theories. }
\label{newt}
\end{table}

\subsection{Fingerprint Invariant of Partitions}\label{subsubfinger}
The {\it fingerprint} invariant is derived from rigid surface operators $(\lambda';\lambda'')$
using the Kazhdan-Lusztig map. It consists of a pair of partitions $[\alpha;\beta]$ associated with the Weyl group conjugacy class.
The  fingerprint invariant of  $(\lambda',\lambda'')$  is constructed  as follows \cite{Wy09}.
\begin{Def}\label{finger}(\textbf{Fingerprint invariant})
\begin{enumerate}
  \item First, add the two partitions  $\lambda=\lambda^{'}+\lambda^{''}$, then construct    the partition $\mu$ as follows:
\begin{equation}\label{mu}
\mu_i=Sp(\lambda)_i=
\left\{ \begin{aligned}
         & \lambda_i + p_{\lambda}(i) \quad  \quad   \textrm{if} \quad \lambda_i  \textrm{ is odd  and} \quad \lambda_i \neq \lambda_{i-p_{\lambda}(i)}, \\
         & \lambda_i   \quad  \quad   \quad   \quad  \quad        \textrm{ otherwise},
       \end{aligned} \right.
\end{equation}
where $p_{\lambda}(i)=(-1)^{\sum^i_{k=1} \lambda_k}$.
  \item Next, define the function $\tau$ from an even positive integer $2m$ to $\pm 1$ as follows.
\begin{itemize}
  \item  For $B_n$($D_n$) partitions,
  \begin{itemize}
    \item   $\tau(2m)=-1$, if  at least one $\mu_i$ such that $\mu_i= 2m$ and any of the following three conditions is satisfied.
\begin{eqnarray}\label{con}
  &&(C1), \quad\quad\quad\quad\quad \,\,\mu_i\neq \lambda_i  \non \\
 && (C2),  \quad\quad\quad\sum^{i}_{k=1}\mu_k \neq \sum^{i}_{k=1}\lambda_k\\
&& (C3)_{SO},  \quad\quad\quad \lambda^{'}_{i} \quad\textrm{is odd}.  \non
\end{eqnarray}
    \item Otherwise,  $\tau(2m)= 1$.
  \end{itemize}
  \item For $C_n$ partitions,  the definition is the same except that the condition $(C3)_{SO}$ is  replaced by
$$(C3)_{Sp},\quad\quad \lambda^{'}_{i}  \quad   \textrm{is even}.$$
\end{itemize}
  \item Finally,  construct  a pair of partitions $[\alpha;\beta]$. For each pair of parts of $\mu$ both equal to $a$ with $\tau(a)=1$,   retain one part $a$ as a part of  the partition $\alpha$.
  For each part of $\mu$ of size $2b$ with $\tau(2b)=-1$,  retain $b$ as a part of  the partition $\beta$.
\end{enumerate}
\end{Def}

\begin{rmk}\label{def}
Note that the part $2m=\lambda_i=\lambda^{'}_{i}+\lambda^{''}_{i}$ corresponds to the height of the $2m$th row in $\lambda$.
\begin{itemize}
  \item $`(3)_{SO} \quad \lambda^{'}_{i} \quad\textrm{is odd}'$  is equivalent to  $`(3)_{SO} \quad \lambda^{''}_{i} \quad\textrm{is odd}'$.
  \item $`(3)_{Sp}\quad \lambda^{'}_{i}  \quad   \textrm{is even}'$ is equivalent to $`(3)_{Sp}\quad \lambda^{''}_{i}  \quad   \textrm{is even}'$.
  \item We introduce some  conventions.
\begin{itemize}
  \item  $\mu=\mu(\lambda)$ stands for the map $\mu$ or the image of the map.
  \item $Ci, \, i= 1,2,3$ are the conditions in formula (\ref{con})  in Definition \ref{finger}.
  \item $p_{\lambda}(i)=(-1)^{\sum^i_{k=1} \lambda_k}$.
\end{itemize}
\end{itemize}
\end{rmk}

According to  formula (\ref{mu}),  $\mu_i\neq\lambda_i$  only happens at the end of a row   with   one box   appended or deleted. The following important lemma can be get  from the definition of fingerprint invariant \cite{SW17-2}(they  will be used frequently.).
\begin{lem}\label{ilem}
Under the map $\mu$, the change of $\lambda_i$th row of the partition $\lambda$ is shown in Table \ref{ssst}, which depends on  the parity of its height and  the sign of $(-1)^{p_\lambda(i)}$.
\begin{table}
  \begin{tabular}{|c|c|l|}
  \hline
  Parity of the height of  row & Sign of  $(-1)^{p_\lambda(i)}$   & Change of the end of a row\\\hline
  odd & $-$ & $\mu_i=\lambda_i-1$ \\\hline
  even & $-$ & $\mu_i=\lambda_i+1$ \\\hline
  even & $+$ & $\mu_i=\lambda_i$ \\\hline
  odd & $+$ & $\mu_i=\lambda_i$ \\
  \hline
\end{tabular}
\caption{ Changes of the end  of the rows in  $\lambda$ under map $\mu$.}
\label{ssst}
\end{table}

These results are visualized in Figs.(\ref{OR}) and (\ref{ER}) \footnote{In this paper, the grey box represents an appended box and the black box represents a deleted box.}. The symbol $'\pm'$ indicate  the sign of $(-1)^{p_\lambda(i)}$ corresponding to $i$th part of $\mu$.
\begin{enumerate}
  \item A box is deleted only at the end of rows with odd heights and the factor $(-1)^{p_{\lambda}(i)} = -1$, as illustrated in Fig.(\ref{OR})(a) and (c).
  \item A box is appended only at the end of rows with even heights and the factor $(-1)^{p_{\lambda}(i)} = -1$, as illustrated in Fig.(\ref{ER})(a).
  \item When a box is appended to a row with even height, the row above it must have a box deleted at the end of row, as illustrated in Fig.(\ref{ER})(a).
  \item When a row with an even height remains unchanged, the row above it also remains unchanged, as illustrated in Fig.(\ref{ER})(b).
\end{enumerate}
\end{lem}
  \begin{figure}[!ht]
  \begin{center}
    \includegraphics[width=3.5in]{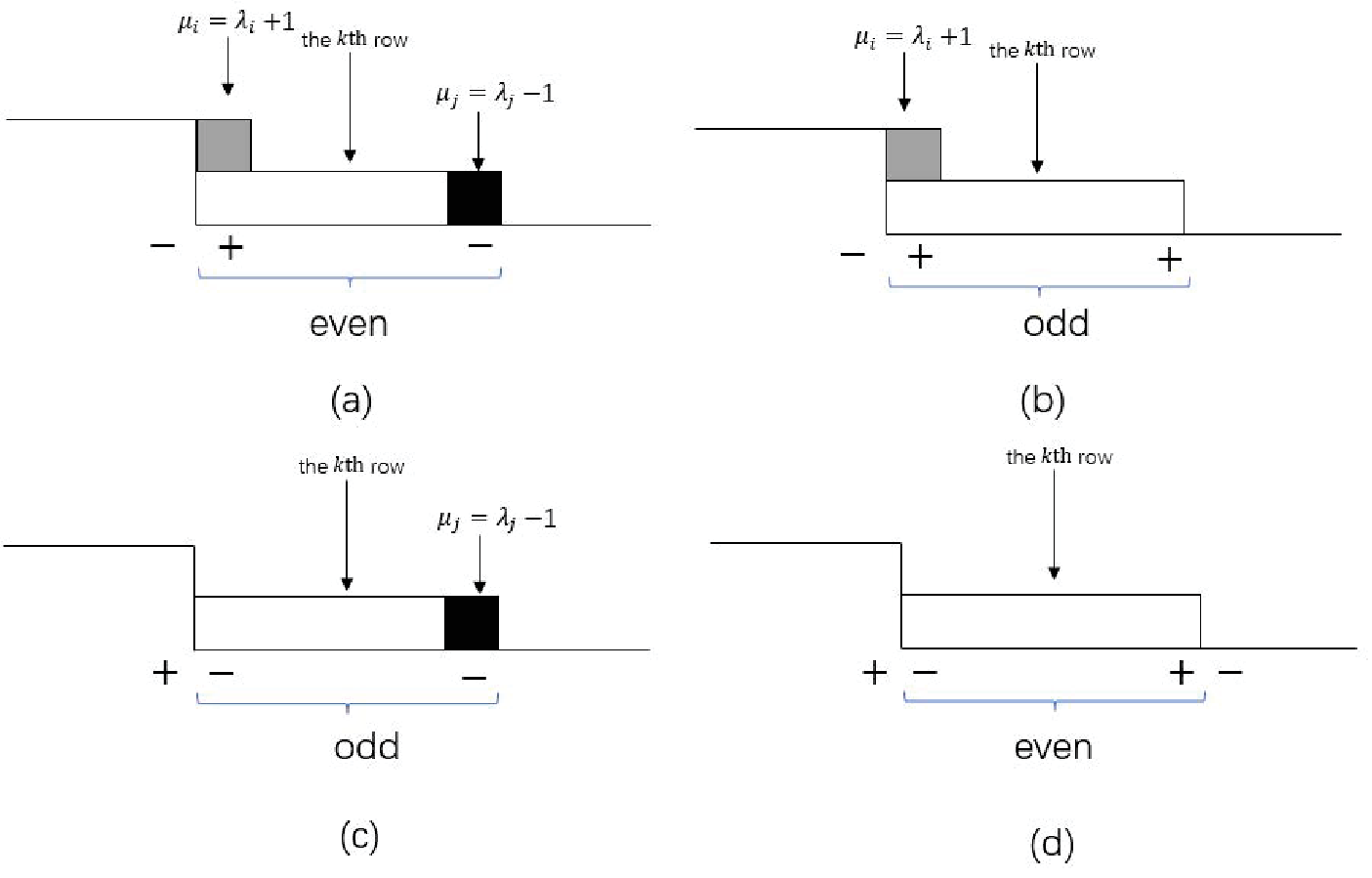}
  \end{center}
  \caption{ Image of  parts $k^{n_k}$($k$ is odd) under the map $\mu$. }
  \label{OR}
\end{figure}
\begin{figure}[!ht]
  \begin{center}
    \includegraphics[width=3.5in]{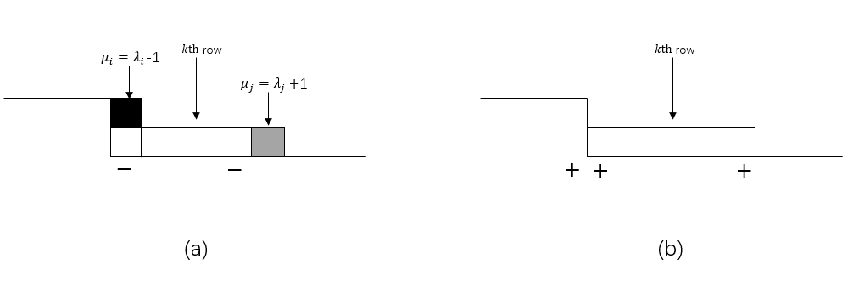}
  \end{center}
  \caption{  Image of parts $k^{n_k}$($k$ is even)under the map $\mu$ . }
  \label{ER}
\end{figure}
\begin{flushleft}
 Lemma \ref{ilem}  is illustrated  by the following example as shown in Fig.(\ref{mumu}),
\end{flushleft}
\begin{ex}
Assume the number of boxes above the $(2k-1)$th row is even. The $2k$th row and  $(2k-1)$th row have different parities, then  a box is deleted  at the end of the $(2k-1)$th row and  a box is appended at the end of the $(2k-2)$th row.
This pattern continues until the    $2i$th row and  $(2i-1)$th row have different  parities.
\begin{figure}
  \centering
  \includegraphics[width=3in]{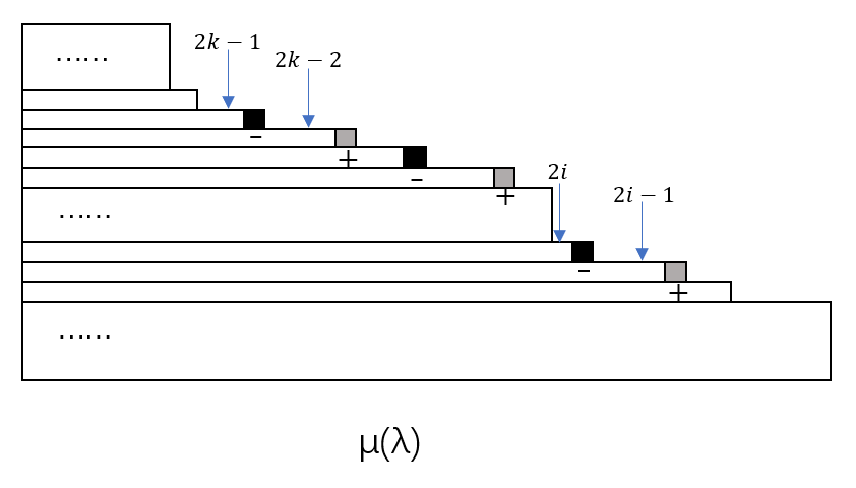}\\
  \caption{$\mu(\lambda)$ is the partition of $\lambda=\lambda^{'}+ \lambda^{''}$ under the map $\mu$. }\label{mumu}
\end{figure}
\end{ex}

From the above lemma, we derive the following facts:
\begin{prop}\label{boxbox}
\begin{enumerate}
  \item   Starting from the highest row, a box is always deleted before a box is appended, as illustrated in Fig.(\ref{mumu}).  
  \item  The deletion of a box must be followed by the appending of a box, ensuring that the number of boxes is balanced, i.e.,
  $$\sum_{k=1}^{i}\mu_k = \sum_{k=1}^{i}\lambda_k$$
at the end of the even height row of $\mu(\lambda)$, as shown in Fig.(\ref{ER}).
  \item Fig.(\ref{OR})(c) is the starting  part of Fig.(\ref{ER})(a). Fig.(\ref{OR})(b) is the ending part of Fig.(\ref{ER})(a).
  \item From Fig.(\ref{ER})(b),  if no change happens to the part $k$ under the map $\mu$, $k$th row of $\lambda$ is equal to the $k$th row of $\mu$.
  \item  From Fig.(\ref{ER}) and  Fig.(\ref{OR}),   for the even and odd rows of $\lambda$, we have
  \begin{equation}\label{mul1}
    length(2m \mathrm{th} \,\mathrm{row} \, \mathrm{of} \,\lambda) \leq length(2m \mathrm{th} \,\mathrm{row} \, \mathrm{of}\, \mu(\lambda)),
  \end{equation}
  and
  \begin{equation}\label{mul2}
   length((2m+1)\mathrm{th} \,\mathrm{row} \, \mathrm{of} \,\lambda)) \geq length((2m+1)\mathrm{th} \,\mathrm{row} \, \mathrm{of}\, \mu(\lambda)).
  \end{equation}
 From Fig.(\ref{ER}) and  Fig.(\ref{OR}), more accurate,    we have
 \begin{equation}\label{mul3}
  |length(k \mathrm{th }\,\mathrm{row} \, \mathrm{of} \,\lambda)- length(k\textrm{th} \, \mathrm{row} \, \mathrm{of}\, \mu(\lambda))|\leq 1.
 \end{equation}
\end{enumerate}
\end{prop}
\begin{rmk}\label{rbox}
\begin{enumerate}
  \item The difference  between $\lambda$ and $\mu$ is less than one box in each row, as shown in Figs.(\ref{OR}) and (\ref{ER}) (or formula ({mul3})). Additionally, the difference  between rows with the same contribution to the symbol is also less than one box per row, as indicated in Table \ref{newt}.
  \item For the above reason and  Propositions \ref{Pb}, \ref{Pd}, and \ref{Pc}, the height of a row in $\lambda$ would not change under both invariants preserving maps.
  \item For the above reason,  we can give each row of $\lambda^{'}$ (or $\lambda^{''}$) a name,  without a  confusion under symbol or fingerprint preserving maps.
\end{enumerate}

\end{rmk}

Note that the fingerprint invariant is  finer than the invariant $\mu$.
\begin{prop}\label{same}
Under the map $\mu$,  both $\mu(\lambda)$   and  $\tau(2m)$  are  invariants   of operators $(\lambda^{'},\lambda^{''})$ with the same fingerprint invariant.
\end{prop}
\begin{proof}
  Let $(\alpha,\beta)$ be the fingerprint invariant  of operators $(\lambda^{'},\lambda^{''})$. According to the calculation of the fingerprint invariant, $\mu(\lambda)$ is constructed by  doubling  parts of $\alpha$ and each part of $\beta$ multiplied by two. Thus it is  the same for operators with the same fingerprint invariant .

  With different values of $\tau(2m)$, the contribution of part $2m$  to the fingerprint invariant is different.  Thus it is  the same for operators with the same fingerprint invariant
\end{proof}

 The following proposition is critical in the proof of the conjecture.
As illustrated in  Figs.(\ref{OR}) and (\ref{ER}), under the map $\mu$, the changes of the odd rows of $\lambda$ are determined by that of the even rows.
\begin{prop}\label{de}
The even rows of a  partition $\lambda$   determine the fingerprint invariant  completely.
\end{prop}
\begin{proof}
According to Lemma \ref{ilem}, the end of the $(2k+1)$th row will lose one box if the end of the $(2k)$th row gains a box. The $(2k+1)$th row remains unchanged if the $(2k)$th row does not change. Thus, the changes in the even rows determine the changes in the odd rows under the map $\mu$.

According to Proposition \ref{same}, $\mu(\lambda)$ is an invariant and does not change under fingerprint-preserving maps.

Moreover, the conditions (\ref{con}) in Definition \ref{finger} need not be considered when we  calculate the partition $\alpha$ for the odd parts of $\lambda$. We only need to consider them  ($\tau(2m)$) of the even parts $2m$ of $\lambda$.
\end{proof}
\begin{rmk}
We give more evidences to this important lemma.
The values of  $\tau(2m)$ are derived from the conditions $Ci$'s. The following three sets give the same fingerprint invariant:
 \begin{enumerate}\label{set3}
   \item Rigid surface operator $(\lambda^{'},\lambda^{''})$.
   \item $\mu(\lambda)$ and $Ci$'s for all even parts.
   \item $\mu(\lambda)$ and $\tau(2m)$ for all even parts.
 \end{enumerate}
\end{rmk}

The following fact is the result of Propositions \ref{boxbox}(1), (2), and (3).
\begin{lem}\label{ccc2}
  If a part $2m$ satisfy the   condition $C2$,  the first  odd part  before it  must satisfy the condition $C1$.
\end{lem}
\begin{flushleft}
The above lemma suggests a possible example where condition $C2$ holds, but condition $C1$ does not.
\end{flushleft}
\begin{ex}\label{example2}
As shown in Fig.(\ref{con2}),  the heights  of the $(j-1)$th,  $(j-3)$th, and  $(j-5)$th rows are even.  The height difference between the $(j-1)$th row and the $(j-3)$th row, as well as between the $(j-3)$th row and the $(j-5)$th row, is $\lambda_i - \lambda_{i+1} =2$ (Note that $\lambda=\lambda^{'}+\lambda^{''}$ does not have to satisfy the rigid constraint.).
As shown in Fig.(\ref{con2}), we have $\mu_{m}=\lambda_{m}$,  which  does not  satisfy the condition $C1$, but it  satisfy the condition $C2$
$$\sum^{m}_{k=1}\mu_k \neq \sum^{m}_{k=1}\lambda_k.$$
Therefore,  the part $\lambda_{m}=j-3$ satisfies the condition $C2$ but does not satisfy the condition $C1$ \footnote{In Section 7.2, we will prove that this case is impossible  for rigid surface operators.  }.
\end{ex}

Immediately, we conclude:
\begin{lem}(\textbf{Rigid constraint})\label{cc2}
For the partition $\lambda$, $\lambda_i-\lambda_{i+1}\leq 1$ is called rigid constraint. Under the rigid constraint of $\lambda$,  the condition  $C1$  implies $C2$.
\end{lem}
\begin{rmk}\label{lemc}
\begin{enumerate}
\item Note that \textit{rigid constraint} is consistent with Rigid condition \ref{rigid} but they are different.
  \item  Since both rows in $\lambda'$ and $\lambda''$ have different lengths, we will not get $\lambda_i - \lambda_{i+1} > 2$ in $\lambda$.
  \item  \emph{Not until} Section \ref{equ}, will we consider   the partition $\lambda=\lambda^{'}+\lambda^{''}$ which has parts satisfy $\lambda_i-\lambda_{i+1} = 2$.
\end{enumerate}
\end{rmk}
\begin{figure}[!ht]
  \begin{center}
    \includegraphics[width=3.5in]{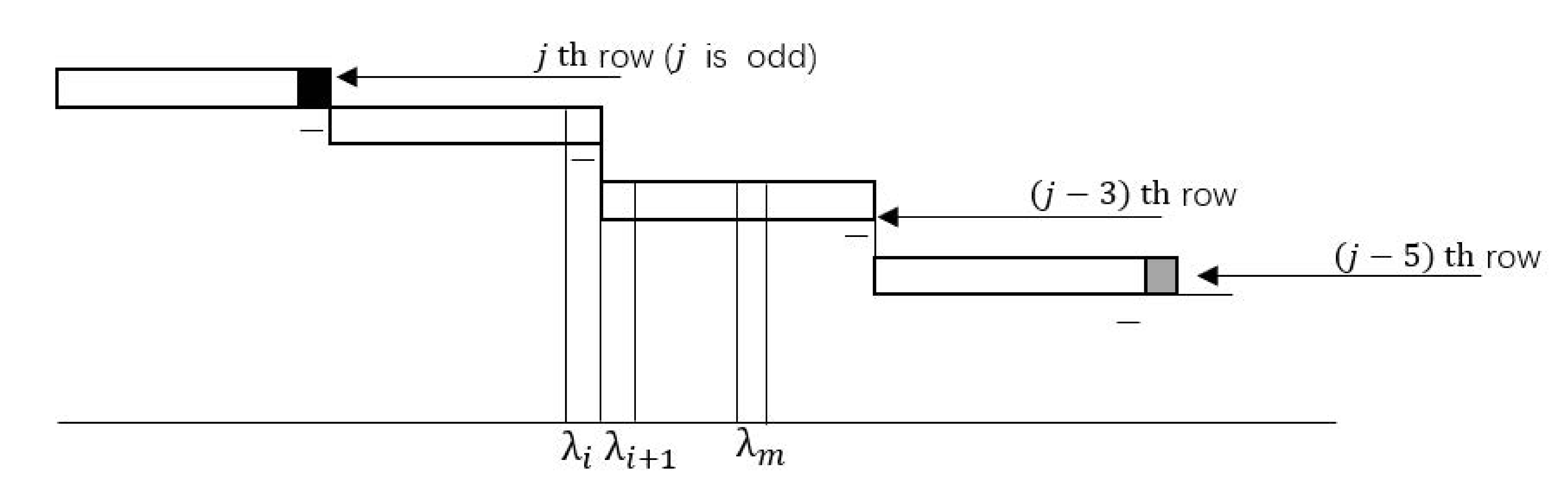}
  \end{center}
  \caption{Example satisfying  only condition $C2$: Even part $\lambda_{i+1}=j-3$ with  $\lambda_i-\lambda_{i+1}=2$.}
  \label{con2}
\end{figure}

\section{Examples }\label{example}
In this section, we introduce examples of maps that preserve symbols and demonstrate that these maps also preserve fingerprint invariants.
These examples provide insights into proving the equivalence of symbol and fingerprint invariants.



\begin{figure}[!ht]
  \begin{center}
    \includegraphics[width=4in]{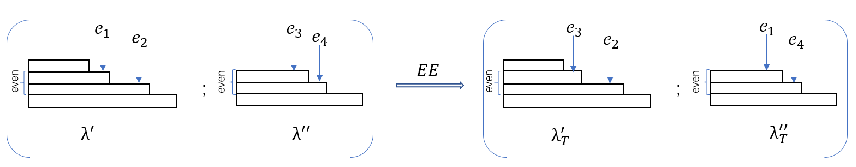}
  \end{center}
  \caption{ Symbol preserving  map $EE: (\lambda^{'},\lambda^{''})\rightarrow (\lambda^{'}_T,\lambda^{''}_T) $,  swapping   row $e_1$ in $\lambda^{'}$ with row $e_3$ in $\lambda^{''}$. Rows $e_1, e_2$ and $e_3, e_4$ are even pairwise rows. Rows $e_3, e_2$ and $e_1, e_4$ are even pairwise rows. }
  \label{eet}
\end{figure}
The first example is illustrated in Fig.(\ref{eet}).  The map $EE$ swaps the top row $e_1$ in an even pairwise rows with the top row $e_3$ in an even pairwise rows.
According to Table \ref{newt},  the contribution to the symbol of each row does not change under the map, thereby  preserving the symbol. Under the map $EE$,   $\lambda$ and $\lambda_T$ are equal as well  as their images under the map $\mu$ as shown in Fig.(\ref{mueet}).
\begin{figure}[!ht]
  \begin{center}
    \includegraphics[width=3.5in]{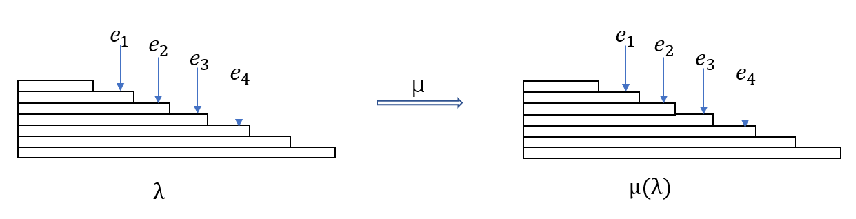}
  \end{center}
  \caption{ $\mu(\lambda)$ is exactly the same with $\mu(\lambda_T)$ under the map $EE$. }
  \label{mueet}
\end{figure}
To calculate the fingerprint, we need to calculate $\tau(2m)$ for each even part  $2m$ in $\mu(\lambda)$, which  is the same for $\lambda$ and $\lambda_T$ as shown in  Table \ref{tmueet}.
Hence, the map $EE$ preserve the fingerprint invariant. Note the part $2m=4$ does not satisfy both conditions $C1$ and $C3$.
\begin{table}[h]
\centering
\begin{tabular}{c|c|c|c}
\hline
$2m$& $C1$ & $C3$ &$\tau{(m)}$\\
\hline
6& N & Y &-1\\
\hline
4& N & Y &-1\\
\hline
2&N &Y& -1\\
\hline
\end{tabular}
\caption{\label{tmueet} Values $\tau{(2m)}$ for  even rows of the partition $\mu(\lambda)$. ‘N':  Condition be not satisfied and ‘Y': condition be satisfied. }
\end{table}


\begin{figure}[!ht]
  \begin{center}
    \includegraphics[width=4in]{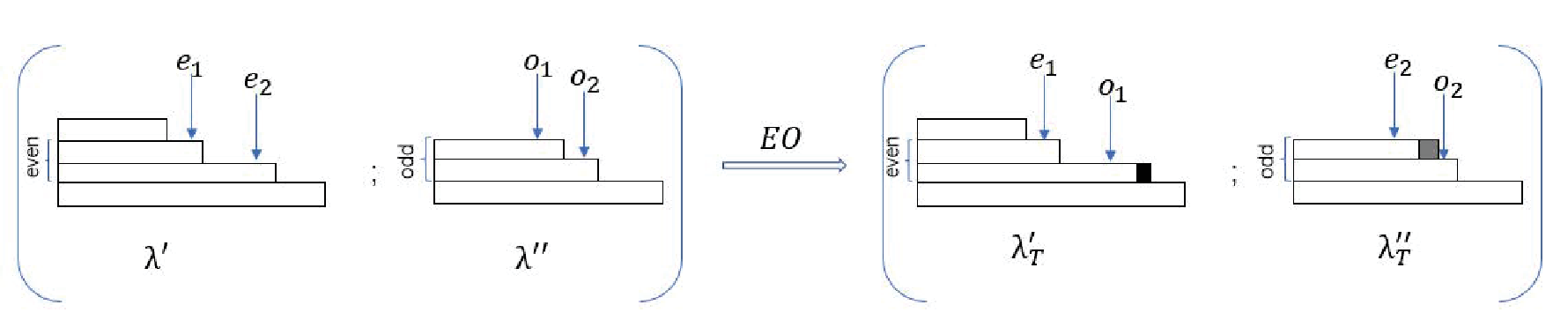}
  \end{center}
  \caption{Symbol preserving  map $EO: (\lambda^{'},\lambda^{''})\rightarrow (\lambda^{T'},\lambda^{T''}) $. It swaps $e_2$ and $o1$, omitting one box at the end of $o1$ and appending one box at the end of $e_2$ with $e_1<e_2<o_1<o_2$.  }
  \label{eto}
\end{figure}
The second example is shown in Fig.(\ref{eto}), the bottom row of an even pairwise rows  of $\lambda^{'}$  swap with the top row of an odd pairwise rows of $\lambda^{''}$. After the map, $o1$ has omitted a box at the end of the row and   $e_2$
is appended a box at the end of row.  According to Table \ref{newt}, the map $EO$ preserves the symbol.

Now we prove the map $EO$  preserve fingerprint.  First we calculate the fingerprint invariant on the left-hand side of the map. The image of $\lambda$ under $\mu$ is given by Fig.(\ref{mueto}).
\begin{figure}[!ht]
  \begin{center}
    \includegraphics[width=3.5in]{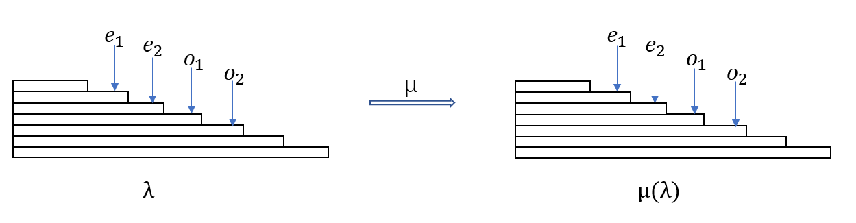}
  \end{center}
  \caption{ Image of  $\lambda $ under the map $\mu$. }
  \label{mueto}
\end{figure}

And the function $\tau(2m)$ for  even parts $m$ in $\mu(\lambda)$ is given by Table \ref{mueto}.
\begin{table}[h]
\centering
\begin{tabular}{c|c|c|c}
\hline
$2m$& $C1$ & $C3$ &$\tau{(m)}$\\
\hline
6& N & Y &-1\\
\hline
4& N & Y &-1\\
\hline
2&N &Y& -1\\
\hline
\end{tabular}
\caption{\label{tmueto} $\tau{(m)}$ for even rows of the partition $\lambda$.  }
\end{table}

\begin{figure}[!ht]
  \begin{center}
    \includegraphics[width=3.5in]{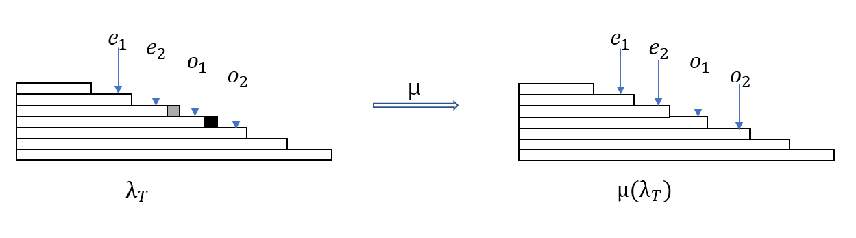}
  \end{center}
  \caption{ Image of  $\lambda_T $ under the map $\mu$. }
  \label{ttmueto}
\end{figure}
Next  we calculate the fingerprint invariant of the partitions  on the right hand side of the map $EO$. The image of $\lambda_T$ under $\mu$ is given by Fig.(\ref{ttmueto}). Note that $\mu(\lambda_T)=\mu(\lambda)$. And the function  $\tau(2m)$ of $\lambda_T$ is  given in Table \ref{mueto2}.
\begin{table}[h]
\centering
\begin{tabular}{c|c|c|c}
\hline
$2m$& $C1$ & $C3$ &$\tau{(2m)}$\\
\hline
6& N & Y &-1\\
\hline
4& Y & N &-1\\
\hline
2&N &Y& -1\\
\hline
\end{tabular}
\caption{\label{mueto2} Values $\tau{(2m)}$ of each even row of the partition $\lambda_T$.}
\end{table}
Compared with Table \ref{tmueto},  the part $2m=4$  satisfies different  conditions with the same $\tau(4)$.

The calculation of the above two examples provides the following insights, which will be rigorously proved in subsequent sections:
\begin{itemize}
  \item The rigid operators $(\lambda^{'},\lambda^{''})$  with the same symbol can be obtained by  combinations of different rows according to Table \ref{newt}.
  \begin{enumerate}
  \item There are combinations of rows that do not change the lengths of the rows, as shown in the first example.
  \item There are combinations of rows that do change the lengths of the rows, as demonstrated in the second example.
  \end{enumerate}
  \item The rigid operators $(\lambda^{'},\lambda^{''})$ with the same fingerprint invariant  correspond to the  parts that satisfy different conditions in formula (\ref{con}).
    \begin{enumerate}
  \item The condition satisfied by the part does not change under symbol preserving maps, as in the first example.
  \item The condition satisfied by the part changes under symbol preserving maps, as in the second example.
  \end{enumerate}
\end{itemize}

\section{Fingerprint  Preserving Maps  Preserve the Symbol Invariant}\label{mapf}
As demonstrated in the previous section, operators with the same invariant can have different configurations. In general, operators with the same symbol or fingerprint can vary significantly, as shown in Appendix B.

To prove the  fingerprint invariant  is equivalent to the  symbol invariant, we need to prove that if two rigid  operators  have the same   fingerprint, they also have the same  symbol and vice versa.
Since operators are composed of rows combined in various ways, operators with the same invariants can differ. Therefore, we consider the contribution to invariants from each row in partitions locally.

For the fingerprint, according to Proposition  \ref{same}, the $\mu(\lambda)$ is an invariant, meaning each row of $\lambda$ reduces to the same row in $\mu(\lambda)$ under the map $\mu$.
Different reductions correspond to different conditions satisfied by  rows of the partition according to formula (\ref{con}). To prove the equivalence of the two invariants, we must  show that these combinations preserve the symbol invariant.

First, we classify the fingerprint preserving maps. According to Proposition \ref{de}, \textit{we only need   to consider the even parts $2m$ of $\lambda$. }
Then, we prove that the classes of maps correspond to different classes of conditions satisfied by the even part. 
According to Lemma \ref{cc2}, condition $C1$  implies condition $C2$ for the partition $\lambda$ with Rigid constraint \ref{cc2}:  $\lambda_i - \lambda_{i+1} \leq 1$.
 Thus, we omit condition $C2$ for the moment until Section \ref{equ}.

 We will prove the conjecture for $B_n$ theory. The proof of the conjecture of other theories will be discussed in Section \ref{equ}.
\subsection{Classification of Fingerprint Preserving Maps}
In this subsection, we derive more results on the fingerprint invariant in preparation for the next two subsections.

The following lemma, derived from Lemma \ref{ilem}, is illustrated in Fig.(\ref{ER}).
\begin{lem}\label{lc13}
Under the map $\mu$, an even part $2m$ of the  partition $\lambda$
\begin{enumerate}
  \item that   satisfy the condition $C1$ implies the  $2m$th row has  been appended a box and   the $(2m+1)$th row has  been deleted a box at the end of the row, as shown in Fig.(\ref{ER})(a).
  \item that  does not satisfy the condition $C1$ implies no change of the  $2m$th row  and   the $(2m+1)$th row, as shown in Fig.(\ref{ER})(b).
\end{enumerate}
\end{lem}
The following lemma is derived from  Definition \ref{finger} and Remark \ref{def}.
\begin{lem}\label{lc13new}
If the  even part $2m=\lambda^{'}_i+\lambda^{''}_i$ does not satisfy the condition $C3$,  then  $\lambda^{'}_i$ is even and  $\lambda^{''}_i$ is also even.
If the  even part $2m=\lambda^{'}_i+\lambda^{''}_i$  satisfies the condition $C3$,  then  $\lambda^{'}_i$ is odd and  $\lambda^{''}_i$ is also odd.
\end{lem}

We have the following critical   proposition.
\begin{prop}\label{c13}
 The part $2m$ can not satisfy both  conditions $C1$ and $C3$.
\end{prop}
\begin{proof}  Assume $t$ is the $2m$th row of $\lambda$ as shown in Fig.(\ref{c1c3})(b).
   Suppose it  satisfies both  conditions $C1$ and $C3$.  Since the   part $2m$  satisfy the condition $C3$,     $\lambda^{'}_i$ is odd and  $\lambda^{''}_i$ is also odd,  according to Lemma \ref{lc13new}.
   Without loss generality,  assume  $t$   is the top row of a pairwise row of $\lambda^{'}$  according to Proposition \ref{Pb},   as illustrated  in Fig.(\ref{c1c3})(a).
   There is an odd number of rows of $\lambda^{''}$  under $t$ in $\lambda$,  according to Proposition \ref{Pd}.

   According to Propositions \ref{Pb} and \ref{Pd},  there are an even 
   number of boxes above $t$ in $\lambda$. Hence, the factor $(-1)^{p_{\lambda}(i)}=1$ at the end of $t$th row, indicating no  change of  the end of $t$ under the map $\mu{(\lambda)}$, according to the third row of Table \ref{ssst}.
   According to Lemma \ref{ilem}(4), the end of the $(2m+1)$th row does not change under the map $\mu$.

As shown in Fig.(\ref{ER})(b),
the part $2m$ does not satisfy the condition $C1$,  which  contradicts   the assumption.
\end{proof}
\begin{figure}
  \centering
  \includegraphics[width=4in]{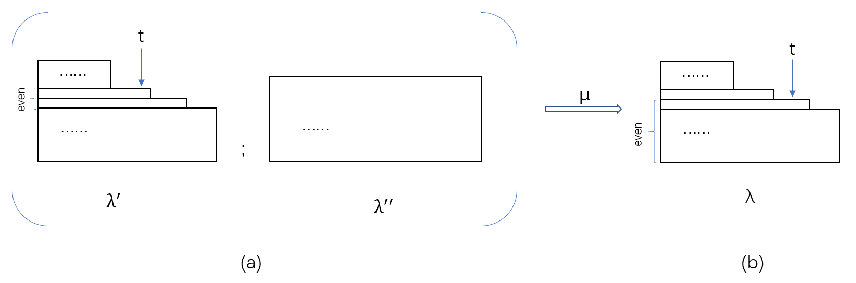}\\
  \caption{ Rigid surface operator $(\lambda^{'}, \lambda^{''})$ with $\lambda=\lambda^{'}+\lambda^{''}$. $t$ is the top row of an even pairwise rows in (a). The second `even'; means the height of the row $t$ is even with value $2m$. }\label{c1c3}
\end{figure}

According to Definition \ref{finger}, the following proposition is evident.
\begin{prop}\label{cn13}
If the  part $2m$ does not  satisfy both   conditions $C1$ and $C3$,  then $\tau{(2m)}=1$.
\end{prop}
\begin{rmk}\label{rc}
The above proposition implies that the even part $2m$ of the partition $\lambda$  either satisfies one of the conditions $C1$ and $C3$ or does not satisfy both.
\end{rmk}

Using Lemma \ref{lc13} and Proposition \ref{c13}, we have
\begin{prop}\label{p123}
\begin{enumerate} An even $2m$ part of $\lambda$ is classified as follows:
  \item The part $2m$  does not satisfy both the  $C1$ condition and $C3$ condition ($\tau(2m)=1$),  corresponding  to Fig.(\ref{ER})(b).
  \item The part $2m$  satisfies the  $C1$ condition($\tau(2m)=-1$),  corresponding to Fig.(\ref{ER})(a).
  \item The part $2m$  satisfies the  $C3$ condition($\tau(2m)=-1$), corresponding to Fig.(\ref{ER})(b).
\end{enumerate}
\end{prop}

Then, we can classify the fingerprint preserving maps as follows.
\begin{prop}\label{ccc}
 There are  two classes of fingerprint preserving maps for the even $2m$ part of $\lambda$:
 \begin{enumerate}
  \item  each case in Proposition \ref{p123} is  preserved.
  \item   Swap of   cases 2 and  3  in Proposition \ref{p123}.
  \end{enumerate}
\end{prop}
\begin{proof}
 For the even part $2m$, the conditions $\tau(2m)=1$ and $\tau(2m)=-1$ do not change under the fingerprint preserving map.
 \begin{itemize}
   \item  When $\tau(2m)=1$,   the part $2m$ does not satisfies both the conditions  $C1$ and  $C3$.
      \item  When $\tau(2m)=-1$, according to Proposition \ref{c13},  the conditions  $C1$ and  $C3$ would  either swap under the fingerprint preserving map or the part $2m$ satisfies the same condition under the map.
 \end{itemize}
 Let  the first class of fingerprint preserving maps  be the maps preserving  the conditions in the definition of fingerprint\footnote{The first class maps   are different from  the identity map as shown in the first example in Section \ref{example}.}.
Let the second class of fingerprint preserving maps  be the maps swapping  the conditions  $C1$ and  $C3$ when $\tau(2m)=-1$.
\end{proof}

The following conclusion is  required.
\begin{prop}\label{fingercc}
The classifications in Propositions \ref{p123} and \ref{ccc}  are complete.
\end{prop}
\begin{proof}
  This is the result of Remark \ref{set3} and  Proposition \ref{c13}.
\end{proof}

\subsection{Fingerprint Preserving Maps Without Change of  Length of Row}
For the first class fingerprint preserving maps in Proposition \ref{ccc}, we have the following proposition.
\begin{prop}\label{lll}
If the  conditions  satisfied by  a even part $2m$ of $\lambda$   does not change under fingerprint  preserving maps, the  length of the  $2m$th  of $\lambda$ and its location in pairwise rows  does not change under fingerprint  preserving maps.
\end{prop}
\begin{proof}
The location of the $2m$ th row of $\lambda$ in a pairwise rows corresponds to condition $C3$. Condition $C3$(satisfied or not)  is preserved  under the first class of fingerprint preserving maps. Thus, the location of the $2m$th row  in pairwise rows  does not change.

Note that we can only  append or delete one box at the end of the $2m$th row under the map $\mu$ according to Lemma \ref{ilem}.
\begin{itemize}
  \item Assume the part $2m$ does not satisfy both the conditions $C1$ and $C3$. Then it also does not satisfy both conditions $C1$ and $C3$ under the first class of fingerprint preserving maps, which means  $\tau{(2m)}=1$.
  The part $2m$ corresponds to Fig.(\ref{ER})(b), according to  Proposition \ref{p123}(1),  which means
  $$\textrm{length}\,\, \textrm{of}\,\, 2m \,\textrm{th}\,\, \textrm{row}\, \,\textrm{of}\,\,  \mu(\lambda)=\textrm{length}\,\, \textrm{of}\,\, 2m \,\textrm{th}\,\, \textrm{row}\, \,\textrm{of}\,\,  \lambda.$$
  If right hand side of the identity  changes under the fingerprint  preserving map, then  the value $\alpha$ in the fingerprint $(\alpha,\beta)$ involving  the the part $2m$ with $\tau{(2m)}=1$ changes, which is a contradiction.
  \item Assume the part $2m$  satisfies  the condition $C1$ only. According to  Proposition \ref{p123}(2),  it corresponds to  Fig.(\ref{ER})(a). 
  If the length of $2m$th row changes under fingerprint  preserving map,  the length of $2m$th row must be appended  one box  to get the same $\mu(\lambda)$, which  corresponds to  Fig.(\ref{ER})(b). To preserve $\tau{(2m)}=-1$,
    the part $2m$ satisfy condition  $C3$. Therefore  the condition satisfied by the part $2m$ changes, which is a contradiction.  
  \item Assume the parts $2m$  satisfy  the condition  $C3$ only. According to  Proposition \ref{p123}(3),  it corresponds to  Fig.(\ref{ER})(b). If the length of $2m$th row changes under fingerprint  preserving map,  the length of $2m$th row must be deleted  one box,  and then we append a box under the map $\mu$ to get the same $\mu(\lambda)$.
  To preserve $\tau{(2m)}=-1$,  the part $2m$ satisfies condition $C1$,  which is a contradiction.
\end{itemize}
We draw the conclusion.
\end{proof}

Consequently, we derive the following result.
\begin{prop}\label{fsfs}
The fingerprint preserving maps that do not alter the conditions in formula (\ref{con}) maintain the symbol invariant.
\end{prop}
\begin{proof}
The contribution to symbol of a row is determined by its location pairwise rows  and the length of row according to Propositions \ref{row-eo} (or Table \ref{newt}).
Then, Propositions  \ref{de} and \ref{lll} directly lead to this result.
\end{proof}

\subsection{Fingerprint Preserving Maps With Change of Length   of  Row}
We now examine the changes of the rows in $\lambda$ under the second class of fingerprint-preserving maps proposed in Proposition \ref{ccc}.

\begin{prop}\label{ll}
Assume the part $2m$ of $\lambda$ satisfies conditions $C1$ or $C3$. Under fingerprint preserving maps that swap conditions $C1$ and $C3$, the length of the $2m$th row of $\lambda$ changes.
\end{prop}
\begin{proof}
Assume the part $2m$ satisfies condition $C1$ but not condition $C3$.
Denote the length of the $2m$th row  of $\lambda$ as $L1$. Then we have
$$L1= \textrm{length}\,\, \textrm{of}\,\, 2m \,\textrm{th}\,\, \textrm{row}\, \,\textrm{of}\,\,  \mu(\lambda)+1$$
according to Proposition \ref{p123}(2).

According to Proposition \ref{ccc},  under the second class of the fingerprint preserving maps,  the part $2m$ satisfies the condition $C3$ but not the condition  $C1$.
 Let length of the $2m$th row of $\lambda$ be $L2$.
  Then we have
$$L2= \textrm{length}\,\, \textrm{of}\,\, 2m \,\textrm{th}\,\, \textrm{row}\, \,\textrm{of}\,\,  \mu(\lambda)$$
according to Proposition \ref{p123}(3).

The partition $\mu(\lambda)$ does not change according to Proposition \ref{same}..
   So we draw the conclusion
    $$L1\neq L2.$$

   Conversely, we can draw the same conclusion.
%
%
%
\end{proof}

According to Lemma \ref{lc13new}, the  change of the condition $C3$ satisfied by the even part $2m$ in $\lambda$ implies the change of the location of $2m$th row
in the pairwise rows.
\begin{prop}\label{llpp}
Assume the part $2m$ of $\lambda$ satisfies the conditions  $C1$ or  $C3$. Under the fingerprint preserving maps,
which swap the  conditions  $C1$ and  $C3$, the   location of $2m$th row
in   pairwise rows  changes.
\end{prop}
\begin{proof}
Let the part $2m$ satisfy the condition $C1$ but not the condition  $C3$. Then the height of the $2m$th row in $\lambda^{'}$($\lambda^{''}$) is odd.
Conversely, if  the part $2m$ satisfy the condition $C3$ but not the condition  $C1$,  the height of the $2m$th row in $\lambda^{'}$($\lambda^{''}$)is even.

Thus  swapping  the  conditions  $C1$ and  $C3$ implies  the   location of $2m$th row  in the pairwise rows  changes according to Propositions \ref{Pb} and \ref{Pd} (or see Fig.(\ref{bdcd})).
\end{proof}

Using Propositions \ref{ll} and \ref{llpp},
we refine the changes in the lengths of the pairwise rows under the fingerprint preserving maps.
\begin{itemize}
  \item For the bottom row of a pairwise rows, which is the $2m$th row in $\lambda$, we have the following result.
\begin{lem}\label{fseob}
Let  $b$  be the  bottom row  of  a pairwise rows. If it is  a $2m$th row in $\lambda$, then it  will be appended to a box under the map $\mu$.
\end{lem}
\begin{proof}
From the assumptions,
the part $2m$ does not satisfy the condition $C3$. According to Proposition \ref{c13},  it must  satisfy the condition $C1$. And then,   according to Fig.(\ref{ER}),  it  will append a box at the end  under the map $\mu$.

%
\end{proof}
\begin{prop}\label{fsfsb}
 Under the fingerprint preserving map,  which change the length of a row,
 the bottom row $b$ of a pairwise rows becomes the top row of a pairwise rows with the length increased by one box.
\end{prop}
\begin{proof}
  According to Propositions \ref{llpp}, the   location of  row $b$  in  pairwise rows  changes, making the  bottom row of a pairwise rows become the top row of a pairwise rows.

   According to Lemma \ref{fseob}, \textit{under the map} $\mu$, the row $b$ is increased by one box.  
  To  preserve  $\mu(\lambda)$,  this change can  only be realized if  the row $b$ is appended with one box after the fingerprint preserving map  because we assume its length changes and the $2m$th in $\lambda$ can not be deleted a box as shown in Fig.(\ref{ER}).
  We prove this by contradiction.
  Assume $b$  decreases  by one box, then under map $\mu$, we have
$$\mu(b-1)\leq length(b)< length(b)+1=\mu(b),$$
where $b-1$ means  row $b$  decreases by one box and $\mu(b)$ is the image of the row $b$ under the map $\mu$..  It is a contradiction.
\end{proof}
  \item For the top row of an even pairwise rows, which is the $2m$th row in $\lambda$, we have the following result.
\begin{lem}\label{fseot}
Let  $t$  be the  top row  of  a pairwise rows. If it is  the $2m$th row in $\lambda$, then it  does not change under the map $\mu$.
\end{lem}
\begin{proof}
From the assumptions,
the part $2m$  satisfies the condition $C3$. According to Proposition \ref{c13},  it does not  satisfy the condition $C1$. And then,   according to Fig.(\ref{ER}),  it  does not change at the end of the row under the map $\mu$.

\end{proof}

Using this lemma, we can prove the following proposition.
\begin{prop}\label{fsfst}
 Under the fingerprint preserving map, which change the length of the row,  the top row of a pairwise rows become the bottom row of a pairwise rows, with the length decreased by one box.
\end{prop}
\begin{proof}
  According to Propositions \ref{llpp}, the   location of  row $t$  in  pairwise rows  changes.
   So the top row of a pairwise rows becomes the bottom row of a pairwise rows.

   According to Lemma \ref{fseot}, \textit{under the map} $\mu$,  row $t$ does not change.  
  To  preserve  $\mu(\lambda)$,  this  can only occur  if row $t$ is deleted  one box after the fingerprint preserving map  because we assume its length changes and the $2m$th row  in $\lambda$ under $\mu$ can not be deleted a box as shown in Fig.(\ref{ER}).
  We prove this by contradiction.
Assume $t$  increases by one box, then under map $\mu$, we have
$$\mu(t+1)\geq length(t)+1>length(t)=\mu(t),$$
where $t+1$ means  row $t$  increases  by one box.  This  is a contradiction.
%
\end{proof}
\end{itemize}

The above results implies the symbol preserving rules.   We have the following Proposition.
\begin{prop}\label{rfs}
  The  fingerprint   preserving maps, which change the length of the row,  preserve the symbol invariant.
\end{prop}
\begin{proof}
We draw the conclusion by using  Propositions \ref{fsfsb} and \ref{fsfst}, which  are just rules in Proposition \ref{fsfsoo}.
\end{proof}
%
%
%
%
%

The preceding proposition complements Theorem \ref{fsfs}. Combining these two results, we have
\begin{thm}\label{fss}
  For rigid surface operators, the  fingerprint  preserving maps  preserve the symbol invariant.
\end{thm}

\section{ Symbol Preserving Maps  Preserve the Fingerprint Invariant.}\label{maps}
According to Proposition \ref{row-eo}, the contribution of each row in a partition to the symbol is invariant.  Hence, given a symbol, there are only a finite number of possible combinations for the rows.

In this section,  we prove that different  combinations of rows for  a symbol invariant   preserve the fingerprint invariant.

\subsection{Classification of   Symbol Preserving Maps}
Proposition \ref{row-eo} is equivalent to the following:
\begin{prop}\label{fsfsoo}
 The contribution to the symbol of the top row of a pairwise rows is equal to that of  the bottom row of a pairwise rows, with the length decreased by one box, and vice verse.  
\end{prop}
\begin{flushleft}
Given the contribution to symbol, if  the length of a row changes, the  pairwise rows which  the row belongs to   have different parities.
\end{flushleft}

Using the above proposition,  we can  immediately classify the symbol preserving maps.
\begin{prop}\label{cc}
 There are  two classes of symbol preserving maps.
 \begin{enumerate}
   \item No change in  the length of the row and its  location in a pairwise rows.
   \item Change in the length of the row and its location in a pairwise rows: The top row of a pairwise rows becomes the bottom row with the length decreased by one box, and vice versa.  
 \end{enumerate}
\end{prop}
The second class of symbol preserving maps corresponds to  Proposition \ref{fsfsoo}.  Under the  symbol preserving map,  if the length of a row does not change, its location in a pairwise rows would not also change according to Table \ref{newt}.
If the length of a row  changes, its location in a pairwise rows  also change. Then we have
\begin{itemize}
  \item The first class of symbol preserving maps corresponds to the change of the length of a row.
  \item The second  class of symbol preserving maps corresponds to without  change of the length of a row.
\end{itemize}

The following conclusion is  required.
\begin{prop}\label{symbolcc}
The classification in Propositions  \ref{cc}  is complete.
\end{prop}
\begin{proof}
  This is the result of Remark \ref{set33} and  Proposition \ref{fsfsoo}.
\end{proof}
\subsection{Symbol Preserving Maps Without Change of    Length of Row}
For the first class symbol preserving maps, we have
\begin{prop}\label{sfoecomplement}
 Let $r$ be the \(2m\)th row of \(\lambda\). Under the symbol preserving maps that without change the length of  row \(r\) in the operator \((\lambda', \lambda'')\), the image of row \(r\) under \(\mu\) does not change.
\end{prop}
\begin{proof}
Firstly, we prove the proposition for a row with even length. Assume \(b\) is the bottom row and \(t\) is the top row of an even pairwise rows of \(\lambda'\) (or \(\lambda{''}\)). According to Proposition \ref{Pb}, the height of \(b\) is even, and that of \(t\) is odd.

\begin{itemize}
\item Let \(r=b\) be the \(2m\)th row of \(\lambda\). Since the height of \(b\) is even in \(\lambda'\), part \(2m\)  does not satisfy  $C3$ condition.
       \begin{itemize}
        \item If it satisfies the \(C1\) condition,  according to Proposition \ref{p123}(2), we must append a box at the end of \(b\) of \(\lambda\) under the map \(\mu\), which means
        \begin{equation}\label{ssn1}
          \textrm{length}\,\, \textrm{of}\,\, 2m \,\textrm{th}\,\, \textrm{row}\, \,\textrm{of}\,\, \mu(\lambda) = \textrm{length}\,\, \textrm{of}\,\, 2m \,\textrm{th}\,\, \textrm{row}\, \,\textrm{of}\,\, \lambda +1.
        \end{equation}
        \item If it does not satisfy the \(C1\) condition,  according to Proposition \ref{p123}(1), we have
        \begin{equation}\label{ssn2}
          \textrm{length}\,\, \textrm{of}\,\, 2m \,\textrm{th}\,\, \textrm{row}\, \,\textrm{of}\,\, \mu(\lambda) = \textrm{length}\,\, \textrm{of}\,\, 2m \,\textrm{th}\,\, \textrm{row}\, \,\textrm{of}\,\, \lambda .
        \end{equation}
      \end{itemize}
 \item Let \(r=t\) be the \(2m\)th row of \(\lambda\). Since the height of \(t\) is odd in \(\lambda'\), part \(2m\) satisfies the \(C3\) condition. Thus, the length of \(t\) in \(\lambda\) does not change under the map \(\mu\) according to Proposition \ref{p123}(3), which means
         \begin{equation}\label{ssn3}
          \textrm{length}\,\, \textrm{of}\,\, 2m \,\textrm{th}\,\, \textrm{row}\, \,\textrm{of}\,\, \mu(\lambda) = \textrm{length}\,\, \textrm{of}\,\, 2m \,\textrm{th}\,\, \textrm{row}\, \,\textrm{of}\,\, \lambda .
        \end{equation}
\end{itemize}
   From the identities  (\ref{ssn1}), (\ref{ssn2}), and (\ref{ssn3}), the length of the image of \(r\) under \(\mu\) does not change under the symbol preserving maps without change of the length  of the $2m$th row .

Similarly, we can prove the proposition for a row with odd length.

%
\end{proof}



\subsection{Symbol Preserving Maps With Change of Length of Row}
We now examine  the conditions satisfied by the even part $2m$  of $\lambda$ under the second class of symbol preserving maps proposed in Proposition \ref{cc}.

First, we present a critical  lemma.
\begin{lem}\label{ssc13}
Under the symbol preserving maps that change the length of the $2m$th row of $\lambda$,  the part $2m$ must  satisfies  either   conditions  $C1$ or $C3$.
\end{lem}
\begin{proof}
According to Proposition \ref{c13}, the conditions $C1$ and $C3$ cannot be satisfied simultaneously by the  part $2m$ of $\lambda$.
Therefore,  the  part $2m$ satisfies   one of the following three cases of conditions:
\begin{itemize}
  \item $C1$ condition.
  \item $C3$ condition.
  \item Neither of $C1$ condition and $C3$ condition.
\end{itemize}

We prove that the third scenario is impossible by contradiction. An  counterexample is shown by the first  map $\mu$ in Fig.(\ref{fssc13}).
  \begin{figure}[!ht]
  \begin{center}
    \includegraphics[width=5in]{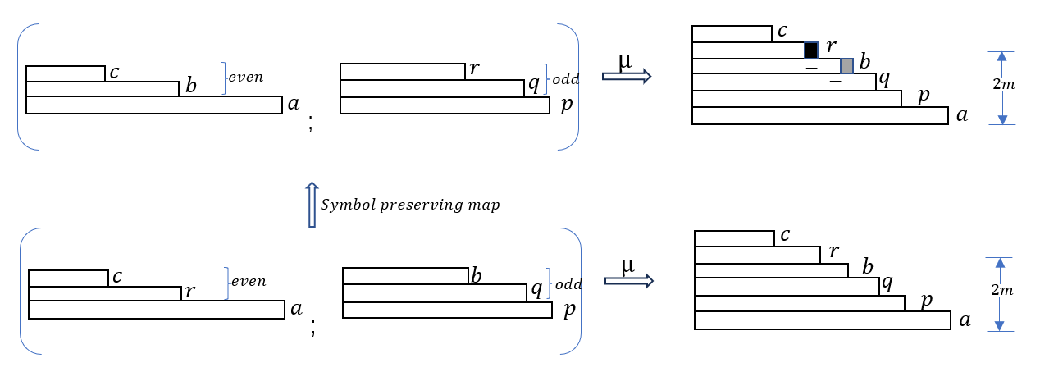}
  \end{center}
  \caption{The $2m$th row is $b$. Row $b$ and $r$ exchange   under the symbol preserving map, changing  the lengths of them. The images of $r$ and $b$ under map $\mu$ are given.}
  \label{fssc13}
\end{figure}

Let  row $b$ be the $2m$th row.
Under the the symbol preserving maps that change the length of the $2m$th row of $\lambda$, row $b$ and $r$ exchange by using the rules in Table \ref{newt}.
Assume row $q$ is the shortest  row of $\lambda^{''}$  under row $b$ in $\lambda$.
The parity of length of pairwise rows $cr$ is different from that of  pairwise rows $bq$.  The parity of length of pairwise rows $cb$ is different from that of  pairwise rows $rq$.

Assume $2m$  does not satisfy  both the  $C1$ condition and the $C3$ condition.
\begin{itemize}
  \item Since part $2m$  does not satisfy  the $C1$ condition,   under the map $\mu$,   the end of the   $2m$th does not change, corresponding to Fig.(\ref{ER})(b)(or Proposition \ref{p123}(1)).
  \item Since it does not satisfy $C3$ condition, there are even number of rows of $\lambda^{'}$ and $\lambda^{''}$ in the first $2m$ rows of $\lambda$    according to Lemma    \ref{lc13new}.
Thus the number of boxes above the $2m$th row $b$ is odd.  Then, under the map $\mu$,   the end of the   $2m$th is appended a box, corresponding to Fig.(\ref{ER})(a).
\end{itemize}
It is a contradiction for the above  results.
 \end{proof}

In Fig.(\ref{fssc13}), the part $2m$ satisfies the condition $C3$ under the bottom $\mu$. Under the   symbol preserving maps that change the length of the $2m$th row of $\lambda$,   the part $2m$ satisfies the condition $C1$ under the top map  $\mu$.   In general, we have
\begin{lem}\label{sfc13}
Under the symbol preserving maps that change the length of the $2m$th row of $\lambda$,  conditions $C1$ and $C3$ swap with each  other.
\end{lem}
\begin{proof}
According to Lemma \ref{fssc13}, the part $2m$ must  satisfy one of the conditions $C1$ or $C3$.

Assume  the part $2m$ satisfies the condition $C1$. Then, the height of  the $2m$th row of $\lambda$ is even in $\lambda^{'}$(or $\lambda^{''}$).
Under  the symbol preserving maps which change the length  of  $2m$th row, the height of the $2m$th row of $\lambda^{'}$ becomes odd in $\lambda^{''}$ according to Proposition \ref{fsfsoo}. Thus, the part $2m$ satisfy the condition $C3$.

Conversely,  assume   part $2m$ satisfy the condition $C3$, then the height of  the $2m$th row of $\lambda$ is odd in $\lambda^{'}$. 
Under  the symbol preserving maps which change the length  of  $2m$th row, the height of the $2m$th row of $\lambda^{'}$ becomes even in $\lambda^{''}$ according to Proposition \ref{fsfsoo}, which means it does not satisfy $C3$.
Thus,  it must satisfies $C1$ by Lemma \ref{fssc13}.
\end{proof}

Using these lemmas, we can draw the following conclusion.
\begin{prop}\label{sfoe}
Let $r$ be the \(2m\)th row of \(\lambda\). Under the symbol preserving maps that change the length of  row \(r\) in the operators \((\lambda', \lambda'')\), the fingerprint invariant is preserved,  which means the images of row \(r\) under \(\mu\) do not change.
\end{prop}
\begin{proof}
Using Lemma \ref{sfc13},  the symbol preserving maps that change the length of a row is the second class of fingerprint preserving maps. They preserve the fingerprint invariant, thus they preserve $\mu$.

\end{proof}

The above proposition complements Proposition \ref{sfoecomplement}. Now, we can draw the following conclusion.
\begin{thm}\label{csf}
For rigid surface operators, the symbol preserving maps preserve the fingerprint invariant.
\end{thm}

\begin{proof}
According to Propositions \ref{sfoe} and \ref{sfoecomplement}, \(\mu(2m)\) does not change under the symbol preserving maps.

 According to Proposition \ref{sfoecomplement}, the even parts of \(\mu\) with \(\tau=1\) do not change under the symbol preserving maps. Thus, the partition \(\alpha\) in the fingerprint invariant \((\alpha, \beta)\), which involves the even parts of \(\mu\), does not change under the symbol preserving maps.

According to the proof of Proposition \ref{sfoecomplement}, the conditions \(C1\) or \(C3\) are preserved under the symbol preserving maps that do not change the lengths of rows in the operators \((\lambda', \lambda'')\).
According to the proof of Proposition \ref{sfoe}, the conditions \(C1\) and \(C3\) swap under the symbol preserving maps that change the lengths of rows in the operators \((\lambda', \lambda'')\).
Thus, the even parts of \(\mu\) with \(\tau=-1\) do not change under the symbol preserving maps, meaning the partition \(\beta\) in the fingerprint invariant \((\alpha, \beta)\) also does not change.

The partition \(\alpha\) in the fingerprint invariant  involves the odd parts of \(\mu\) are determined by  even parts of $\mu$ according to proposition \ref{de}.

Combining all results, the fingerprint invariant \((\alpha, \beta)\) does not change under the symbol preserving maps.
\end{proof}

\section{Equivalence of  Symbol  Invariant and  Fingerprint Invariant }\label{equ}
 In preceding sections, we restricted the partition $\lambda$ to satisfy   the condition $\lambda_i-\lambda_{i+1} \leq 1$ for all $i$  in $\lambda$.
Due to this restriction,   condition $C2$   can be omitted according to Lemma \ref{cc2}. 
Under this constraint, the equivalence of symbol invariant and fingerprint invariant is proved according to Theorems  \ref{fss}  and  \ref{csf}.

In this section, we  relax the constraints   and  discuss  the condition $C2$  further. Then, we  prove the conjecture in the general case.

The proof the conjecture for $C_n$ and $D_n$ theories are discussed.
\subsection{Rigid Constraint: $\lambda_i-\lambda_{i+1} \leq 1$}
The constraints $\lambda_i-\lambda_{i+1} \leq 1$ are termed  the rigid constraint, which is consistent with the definition of Rigid condition \ref{rigid}.
Subject to rigid constraint, combining  Theorems \ref{fss} with Theorem \ref{csf},   we draw the following conclusion.
\begin{thm}
For rigid surface operators, the symbol invariant  is  equivalent to the fingerprint invariant  under the constraints $\lambda_i-\lambda_{i+1} \leq 1$ for all $i$ in  $\lambda$.
\end{thm}

\subsection{No Rigid Constraint}
If there are two equal rows in $\lambda$, such that  $\lambda_i-\lambda_{i+1} =2$ in Remark \ref{lemc},   we would  consider  the condition $C2$  in Definition \ref{con} as explained in Fig.(\ref{con2}) of  Example \ref{example2}. Now,  we discuss the influences of the occurrence  of such  equal two rows in $\lambda$ on  the   proofs  in Sections \ref{mapf} and \ref{maps}.
%
%
%
%

\begin{flushleft}
\textbf{Fingerprint invariant implies symbol invariant (Sections \ref{mapf}):} 
\end{flushleft}
%

Surprisingly,  we have
\begin{lem}\label{subfinger}
 For  rigid surface operator  $(\lambda^{'}, \lambda^{''})$,  the part $\lambda_i=2m$ of $\mu(\lambda)$ with $\tau(2m)=-1$ and  $\lambda_{i}-\lambda_{i+1} =2$ in $\lambda$ can not satisfy the $C2$ condition.
\end{lem}
\begin{proof}The $2m$th row and $(2m-1)$th row of $\lambda$ with equal length must belong to  different partitions of    $(\lambda^{'}, \lambda^{''})$. And they must be in the same position of   pairwise rows  according to the definition of rigid $B_n$ surface operator: Both the heights of the $2m$th row and $(2m-1)$th row of $\lambda$ should have the same parity.
Thus, there is an even number of boxes above the $2m$th row in  $\lambda$, which means  $(-1)^{p_\lambda(i)}=1$ at  the end of the $(2m+1)$ row of $\lambda$ corresponding to Fig.(\ref{OR})(b) or (d) {\footnote{Hence,  no box  can be deleted at the end of the $j$th row in  the example in Fig.(\ref{con2}), rendering this example impossible.}}.
At the end of the $(2m+1)$th row, we have  $\sum \mu_k =\sum \lambda_k$   according to Proposition \ref{boxbox}(2).
 Thus,  part $2m$ corresponds to   Fig.(\ref{ER})(b)
 which means it can not  satisfy the condition $C2$.
\end{proof}

This lemma means  the condition $C2$ can be omitted in the definition of the fingerprint invariant for \textit{rigid surface operator} \footnote{ However,  the condition can not omitted for general operators according to  the proof of this lemma.}.
We analyze this problem furthermore, considering two cases for rows of equal length:
\begin{itemize}
   \item Rows in the same positions of  pairwise rows: According to  the proof of Lemma \ref{subfinger},  the ${2m}$th and ${(2m-1)}$th rows of $\lambda$  are equal to  the ${2m}$th and ${(2m-1)}$th rows of $\mu$, which means the lengths and the positions  in pairwise rows of these   two rows are preserved under the fingerprint preserving maps.
  \textit{ Then their contributions to symbol invariant are   preserved under the fingerprint preserving maps.}
   \item    Rows in  different  positions of  pairwise rows(According to Lemma \ref{subfinger}, in fact,  it is impossible.):  Then the heights of these equal rows in $\lambda$ are odd, which can be omitted in the discussions  according to Proposition \ref{de}.
\end{itemize}

\begin{flushleft}
\textbf{Symbol invariant implies fingerprint invariant(Sections \ref{maps}):} From the view of the symbol invariant, we would get a more clearer picture.
\end{flushleft}
According to the  arguments in Lemma \ref{subfinger}, the $2m$th row and $2m-1$th row in $\lambda^{'}$ and $\lambda^{''}$, respectively, must be in the same position of   pairwise rows with the same parity.
Thus, these two equal rows  do not change under symbol preserving maps according to Table.\ref{newt}.  Then the number of boxes of $\lambda$ above these two rows is fixed under the symbol preserving maps, \textit{which means they preserve the fingerprint invariant. }


In summary,
\begin{enumerate}
\item Condition $C2$ in the definition of fingerprint invariant can be omitted for rigid surface operators.
  \item Two equal rows of $(\lambda^{'}, \lambda^{''})$  equal to the image under the map $\mu$. Thus, they do not change under fingerprint preserving maps.
  \item Two equal rows of $(\lambda^{'}, \lambda^{''})$   have the same position in pairwise rows and the same length under  symbol preserving maps.  Thus, they do not change under symbol preserving maps.
\end{enumerate}
No  changes of Two equal rows  under  invariants preserving maps, which means both of these maps belong to the first class maps. And these invariants preserving maps preserve each other.
The presence of these two equal rows in $\lambda$ does not affect the proofs of Theorems \ref{fss} and \ref{csf} where we have  not used     the condition $C2$.
Hence   we prove the conjecture that
\begin{thm}\label{bt}
The symbol invariant  is  equivalent to the fingerprint invariant  for  rigid surface operators in the $B_n$ theory.
\end{thm}

\subsection{$C_n(Sp(2n))$ and $D_n(SO(2n+1))$ theories}
For a rigid  operator $(\lambda^{'},\lambda^{''})$ in $B_n$ theory,   $\lambda^{'}$  and   $\lambda^{''}$ are elements of  $B_n$ and  $D_n$ theories, respectively.
In this subsection, we will prove Theorem \ref{bt} for rigid operators in  $D_n$  and  $C_n$ theories.
\begin{flushleft}
$D_n$ \textbf{theory:}
\end{flushleft}
For a rigid  operator $(\lambda^{'},\lambda^{''})$ in $D_n$ theory,  both $\lambda^{'}$  and  $\lambda^{''}$ are elements in $D_n$ theory.
To  calculate   the fingerprint invariant, we compute the invariant $\mu$ of $(\lambda^{'},\lambda^{''})$ using the same conditions as  the operators in $B_n$ theory according to  Definition \ref{finger}.
For symbol invariant,  since the rules in Table \ref{newt} are independent of theories, we use Table \ref{newt} as  the operators in $B_n$ theory according to  Definition \ref{D2}.
Therefore, for the $D_n$
  theory, we can draw the same conclusions using the same strategy as for the $B_n$
  theory without modifications.

\begin{flushleft}
$C_n$ \textbf{theory:}
\end{flushleft}
For a rigid  operator $(\lambda^{'},\lambda^{''})$ in $C_n$ theory,  we have $\lambda^{'}$  in $C_n$ theory and $\lambda^{''}$  in $C_n$ theory.
For the calculation of the fingerprint invariant, we compute the invariant $\mu$ of $(\lambda^{'},\lambda^{''})$ using the same conditions as  the operators in $B_n$ theory  except that the condition
 $$(3)_{SO} \quad \lambda^{'}_{i} \quad\textrm{is odd}$$
  changes to
  $$(3)_{Sp}\quad \lambda^{''}_{i}  \quad   \textrm{is even}$$
according to  Definition \ref{finger}.
For the calculation of the symbol invariant, we use Table \ref{newt} as   the operators in $B_n$ theory according to  Definition \ref{D2}.
Thus, for $C_n$
  theory, we can draw the same conclusions with minor modifications by changing the condition $(3)_{SO}$
  to $(3)_{Sp}$
  in the arguments.

In summary, the top row or  bottom row of a pairwise rows has the same parities of height  for the partitions in the rigid surface operator $(\lambda^{'},\lambda^{''})$ in the $B_n$,  $C_n$,  and $D_n$ theory.
 Thus, we can prove the conjecture for all theories by  using  the  procedures in $B_n$ theory with minor modification.
Finally, conclude
\begin{thm}\label{fff}
The symbol invariant  is equivalent  to the fingerprint invariant  for the rigid surface operators in the $B_n$, $C_n$, and $D_n$ theories.
\end{thm}

\section{Conclusions}\label{summary}
The rigid surface operators are described by the partitions pair $(\lambda^{'}, \lambda^{''})$ \cite{GW08}. In this paper, we have  proven the conjecture proposed in \cite{Wy09} that the fingerprint invariant is equivalent to the symbol invariant for rigid surface operators, leading to the result
\begin{flushleft}
\textbf{ `Theorem 7.3: The symbol invariant  is equivalent  to the fingerprint invariant  for the rigid surface operators in the $B_n$, $C_n$, and $D_n$ theories.'}
\end{flushleft}
The examples in Section \ref{example} provide significant insights into the proof of the equivalence of these two invariants: Type of invariants preserving maps involve the change of the length of row.
The fingerprint preserving maps,  which change the length of a row, relate to   the conditions $Ci$ satisfied by $2m$ part in $\lambda$.

 Proposition \ref{de} and Lemma \ref{subfinger} are crucial for proving the conjecture.
 \begin{itemize}
   \item Lemma \ref{subfinger} implies that condition $C2$ in the definition of the fingerprint invariant for rigid surface operators is redundant and can be omitted.
Thus, calculating the fingerprint invariant needs to consider only conditions $C1$ and $C3$, which encompass all the information about  $\lambda'$ and $\lambda''$.
   \item  Proposition \ref{de} indicates that the even rows of the partition $\lambda$ encode all the useful information of the fingerprint invariant of rigid surface operators.
This proposition allow  us to focus on the even rows of $\lambda$ and neglect the odd ones,   reducing  the complication of the classification of the fingerprint persevering maps as well as the entire proof.
 \end{itemize}

Section \ref{mapf} presents the first part of the proof, establishing that the fingerprint invariant implies the symbol invariant for rigid surface operators, as summarized in Table \ref{cfinger}. This table focuses on the changes in the $2m$th row of $\lambda$.
The table uses the following notations:
\begin{itemize}
  \item $F(S)$: Fingerprint invariant or symbol invariant.
  \item $C(N)$: Change of the length of  row  or no change.
  \item $1(-1)$: Values of $\tau(2m)=1(-1)$.
\end{itemize}
For example, $FN_1$ denotes fingerprint invariant preserving maps with a length change and $\tau(2m)=1$.
The propositions in brackets provide the conclusions.
The first
columns list the classifications of the fingerprint invariant preserving maps,
the second
column specifies the conditions satisfied, the third column describe the change of the length of row,  the fourth column identifies the location in the pairwise rows, and the fifth column references the theorems providing the proof.
The classification of fingerprint-preserving maps is on the left of the double vertical line, while the right side demonstrates that these maps also preserve the symbol invariant.
\begin{table}
\begin{tabular}{|l|l||c|c|l|}\hline
Type & Conditions:$Ci$ &Length($2m$th row)&  Location  & Theorem \\ \hline
$FN_1$ & Neither(P\ref{ccc}(1))&No change(P\ref{lll})  &   No change(P\ref{lll}) & P\ref{fsfs} \\ \hline
$FN_{-1}$ & $C1$ or $C3$ (P\ref{ccc}(1))  &No change(P\ref{lll})  & No change(P\ref{lll}) & P\ref{fsfs}\\ \hline
$FC_{-1}$& $C1\leftrightarrow C3$(P\ref{ccc}(2))& Change(P\ref{ll})  & $b\leftrightarrow  t$(P\ref{llpp})& P\ref{rfs} \\ \hline
\end{tabular}
\caption{Fingerprint invariant implies the symbol invariant.}
\label{cfinger}
\end{table}

In Section \ref{maps}, the second part of the proof, we demonstrate that the symbol invariant implies the fingerprint invariant, as summarized in Table \ref{csymbol}.
The left side of the double vertical line classifies the symbol-preserving maps, while the right side shows that these maps also preserve the fingerprint invariant.
The classification of the symbol preserving maps is given by Proposition \ref{cc}, which  is based on  Proposition \ref{fsfsoo}.
Proposition \ref{fsfsoo} is a concise form of Table \ref{newt}, derived from the construction of the symbol invariant \cite{Shou-sc}.

\begin{table}
\begin{tabular}{|l|l|c||c|c|}\hline
Type&Length($2m$th row) &  Location&   Conditions:$Ci$ & Theorem \\ \hline
$SN_{1}$&No(P\ref{cc}(1)) & No(P\ref{cc}(1))&   Neither (P\ref{sfoecomplement}) & P\ref{sfoecomplement}\\ \hline
$SN_{-1}$&No(P\ref{cc}(1)) & No(P\ref{cc}(1))  &  $C1$ or $C3$ (P\ref{sfoecomplement})  & P\ref{sfoecomplement}\\ \hline
$SC_{-1}$&Change(P\ref{cc}(2))& $b\leftrightarrow  t$(P\ref{cc}(1)) & $C1\leftrightarrow C2$(Lem\ref{sfc13})  & P\ref{sfoe} \\ \hline
\end{tabular}
\caption{Symbol invariant implies the fingerprint invariant.}
\label{csymbol}
\end{table}

 From Tables \ref{cfinger} and \ref{csymbol}, we have:
 \begin{itemize}
 \item Note that  we focus exclusively on the even rows in according to  Proposition  \ref{de}.
   \item The classification of  symbol preserving  maps and  fingerprint preserving map is complete,  according to  Propositions \ref{symbolcc} and \ref{symbolcc}. The change of the $2m$th row relates to its location  in a pairwise rows in rigid surface operator and the conditions $C1$ and $C3$ satisfied by part $2m$ in $\lambda$.
   \item The second table is a   reverse process of the first one, demonstrating their consistency: The exchange of conditions $C1$ and $C3$ satisfied by part $2m$  relates to  the change of the position in a pairwise rows of the $2m$th row in the second class maps.
   The conditions $Ci$ and the locations are preserved in the first class maps.
 \end{itemize}

 The proof of the conjecture involves considering the influence of even rows of a rigid surface operator $(\lambda^{'}, \lambda^{''})$, which is `local'.
In Appendix A, we  describe another strategy to prove that the symbol invariant implies the fingerprint invariant. This approach is based on the global classification of  symbol preserving maps, differing from the approach in Proposition \ref{cc}.
Since we cannot use Proposition \ref{de}, the proof becomes more complex and less feasible.

Clearly, more work can be done.
Given a fingerprint invariant or a symbol invariant,  there are many rigid surface operators with the same invariant as shown in Appendix A.
It would be interesting to  find  canonical-like element among them.
Since the invariants are equivalent, it would be interesting to finding a formula to calculate one invariant from the other one. Hopefully our constructions
will help make further progress.


In addition, the Weyl group of a simple Lie algebra is of particular importance. Both its conjugacy classes and unitary representations are parameterized by partitions but there is no canonical isomorphism between them.
The  Kazhdan-Lusztig map  is a map from the rigid surface operator to the set of conjugacy classes of the Weyl group (fingerprint invariant).
The Springer correspondence  is a map from a rigid surface operator to a unitary representation  of the Weyl group (symbol invariant). For rigid surface operators,  we have proved the equivalence of the two invariants.
Then it would be interesting to finding  a canonical isomorphism between them using the canonical elements of the two invariants for rigid surface operators.
 The proof  should shed light on further studies.

\begin{center}
\textbf{Data Availability Statement}
\end{center}
The data that support the findings of this study are available from the corresponding author, [BaoShou], upon reasonable request.

\section*{Acknowledgments}
We would like to thank   Xiaoman Luo  for  many helpful discussions.
Chuanzhong Li is supported by  National Natural Science Foundation of China (Grant No.12071237).

\appendix
\section{Classification of Maps Preserving Symbol}\label{cla}
In this section, we give a `global' classification of the symbol preserving maps,   distinct from that in Section \ref{maps}.
There are two fundamental types of symbol-preserving maps: $S$ type and $D$ type.
The symbol preserving maps in this section are constructed  using Table \ref{newt} or Proposition \ref{row-eo}. We assume the  surface operators are in the  $B_n$ theory.

\begin{flushleft}
\textbf{{Take odd rows from one partition to another partition:}}
\end{flushleft}
 As shown in Figs.(\ref{te}) and (\ref{to}),  
 the parity of the length of  row $a$ and the rows above it are all  even.
The parities of the lengths of the last several  rows of $\lambda^{''}$  are odd.

 Assume row $a$ is  the bottom  row of an even  pairwise rows. It is inserted into the partition  $\lambda^{''}$ as the bottom row of an even pairwise rows.
To preserve the symbol,   the  parities of the lengths and heights of the rows above $a$ in $\lambda^{''}$ would change under the map $TE$, as well as that of the rows above $a$ in  $\lambda^{''}$.
\begin{figure}[!ht]
  \begin{center}
    \includegraphics[width=4in]{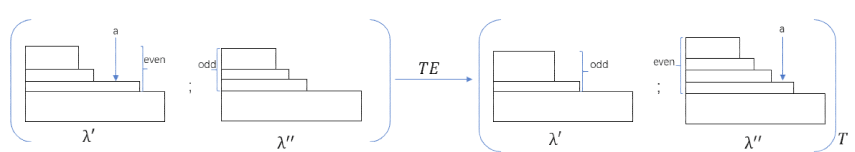}
  \end{center}
  \caption{Row $a$ is inserted  into $\lambda^{''}$ under the map $TE$. }
  \label{te}
\end{figure}

  As shown in Fig.(\ref{to}),  row $a$  is inserted into the partition $\lambda^{''}$ as the  top row of an odd pairwise rows. To preserve the symbol, the parities of the lengths and heights of the rows above $a$ of  $\lambda^{'}$  change under the map  $TO$, making  the rows above $a$ of the  partition $\lambda^{''}$  even.
\begin{figure}[!ht]
  \begin{center}
    \includegraphics[width=4in]{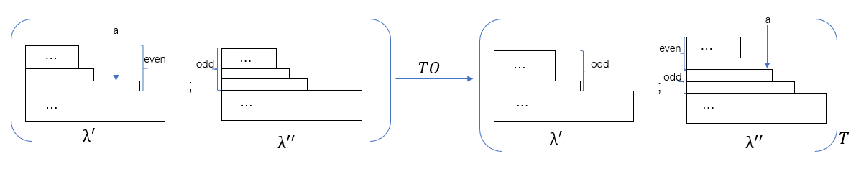}
  \end{center}
  \caption{The row $a$ is inserted   into $\lambda^{''}$ under the map $TO$. }
  \label{to}
\end{figure}

Similarly,  an odd row from   $\lambda^{'}$  can be inserted into $\lambda^{''}$,   leading to two additional   two cases. \footnote{We would not consider    maps suffering  from   severe constraints, as shown in Figs.(\ref{te}) and (\ref{to}).}

\begin{figure}[!ht]
  \begin{center}
    \includegraphics[width=4in]{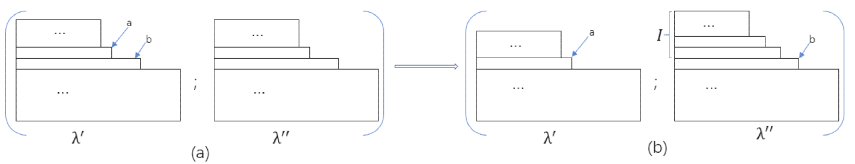}
  \end{center}
  \caption{Insertion of  the row  $b$ of $\lambda^{'}$ into $\lambda^{''}$. }
  \label{s1}
\end{figure}
To preserve the symbol, inserting a row from $\lambda^{'}$ into $\lambda^{''}$ changes the parities of the lengths and heights of all rows above the insertion position. Inserting another shorter row halts these changes.
The rows $a$ and $b$ are of pairwise rows. We insert the row $b$ of $\lambda^{'}$ into $\lambda^{''}$ as shown in Fig.(\ref{s1})(a). To  preserve the symbol, the parities of lengths and heights of   the rows in the region $I$  change  similar  to the maps $TE$ and $TO$, as well as the rows above the row $b$ of  $\lambda^{'}$.
Next, we insert  row $a$ into $\lambda^{''}$ as shown in Fig.(\ref{s2}).  The parities of lengths and heights of the  rows in the region $III$  change  again, recovering  the original parities of   lengths and heights as  in Fig.(\ref{s1})(a), as well as those of the rows above $a$ of $\lambda^{'}$.
If    rows $a$ and $b$ are not of a  pairwise rows, the same result can be obtained with minor modifications.
\begin{figure}[!ht]
  \begin{center}
    \includegraphics[width=4in]{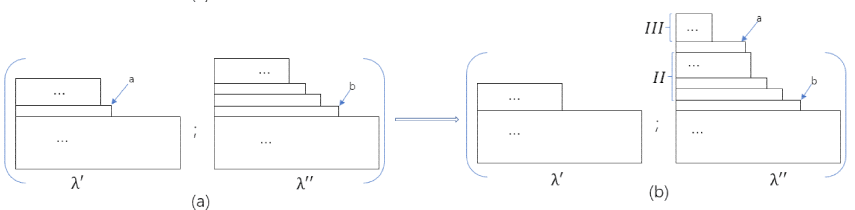}
  \end{center}
  \caption{Insertion of  the row $a$ into $\lambda^{''}$. }
  \label{s2}
\end{figure}

\begin{figure}[!ht]
  \begin{center}
    \includegraphics[width=4in]{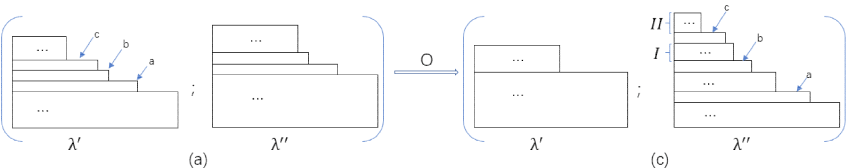}
  \end{center}
  \caption{Insertion of  the rows $a$, $b$, and  $c$ into $\lambda^{''}$. }
  \label{o}
\end{figure}
Inserting three rows from $\lambda^{'}$
 into $\lambda^{''}$
  can be decomposed into two independent fundamental maps.
For a rigid  operator as shown in Fig.(\ref{o})(a), we insert three rows $a$, $b$, and $c$  into $\lambda^{''}$. To  preserve the symbol,  the parities of the lengths of all the rows in the region $II$ do  not  change under the map $O$.   the parities of lengths and heights of all the rows in the region $I$ would   change  according to the maps   $TE$ and $TO$ as shown in Fig.(\ref{te}) and Fig.(\ref{to}).  Then the map $O$ can be decomposed into two maps $O1$ and $O2$ as shown in Fig.(\ref{o1}) and Fig.(\ref{o2}).
\begin{figure}[!ht]
  \begin{center}
    \includegraphics[width=4in]{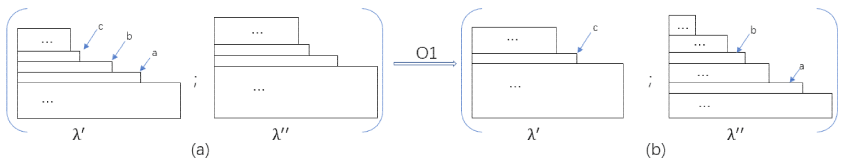}
  \end{center}
  \caption{Insertion of  the rows $a$ and $b$   into $\lambda^{''}$. }
  \label{o1}
\end{figure}
\begin{figure}[!ht]
  \begin{center}
    \includegraphics[width=4in]{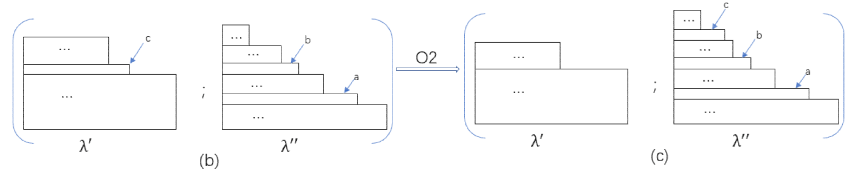}
  \end{center}
  \caption{Insertion of the rows   $c$ into $\lambda^{''}$. }
  \label{o2}
\end{figure}

\begin{flushleft}
\textbf{{Take even rows from one partition to another partition:}}
\end{flushleft}

\begin{figure}[!ht]
  \begin{center}
    \includegraphics[width=3in]{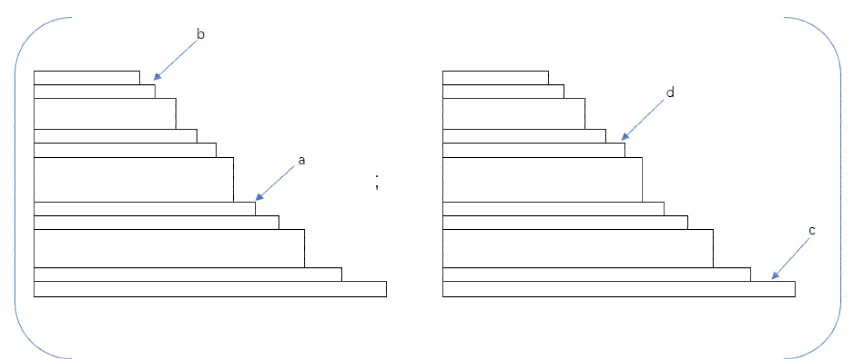}
  \end{center}
  \caption{Partitions  $\lambda^{'}$ and $\lambda^{''}$ with  lengths $L(b)>L(d)>L(a)>L(c)$. }
  \label{abcd}
\end{figure}
Next, we consider the scenario where two rows from $\lambda'$
  are inserted into $\lambda''$
  and vice versa.
 As shown in Fig.(\ref{abcd}), $a$ and $b$ are rows of  $\lambda^{'}$  and $c$ and $d$ are rows of  $\lambda^{''}$,   with  lengths $L(b)>L(d)>L(a)>L(c)$. The insertion  of   rows $a$ and $b$ into $\lambda^{''}$ and   rows $c$ and $d$ into $\lambda^{'}$ simultaneously can be decomposed into two separate mappings.
 One mapping involves inserting row $a$ into $\lambda''$
  and row $c$ into $\lambda'$
  simultaneously, and the other mapping involves inserting row $b$ into $\lambda''$
  and row $d$ into $\lambda'$
  simultaneously.
\begin{figure}[!ht]
  \begin{center}
 \includegraphics[width=3in]{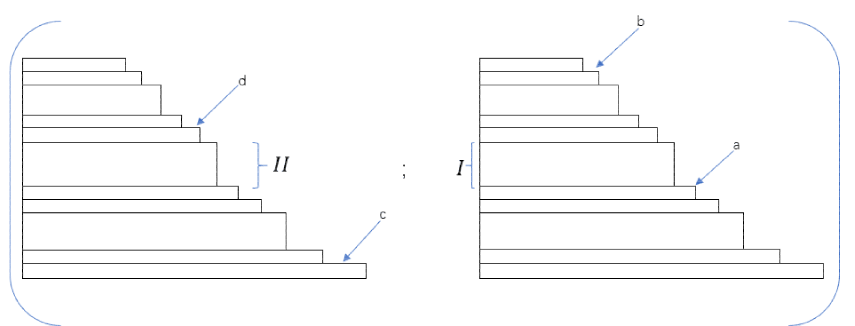}
  \end{center}
  \caption{Insertions of  the rows $a$ and $b$ into $\lambda^{''}$ and insertion of  the rows $c$ and $d$ into $\lambda^{'}$ happen at the same time  at the same time. Regions $I$ and $II$ are the big rectangle in the middle of $\lambda^{'}$ and $\lambda^{''}$, respectively. }
  \label{ab}
\end{figure}

The insertions  of  the rows $a$ and $b$ into $\lambda^{''}$ and  the rows $c$ and $d$ into $\lambda^{'}$ at the same time are  shown in Fig.(\ref{ab}). The parities of lengths and heights of the rows in the region $I$ and $II$  do not change   under this map.
So we can insert   the row  $a$ into $\lambda^{''}$ and insert the row $c$ into $\lambda^{'}$ at the same time firstly as shown in Fig.(\ref{ab1}). And then  insert   the row  $b$ into $\lambda^{''}$ and insert the row $d$ into $\lambda^{'}$ at the same time.
Or we can insert   the row  $b$ into $\lambda^{''}$ and insert the row $d$ into $\lambda^{'}$ at the same time firstly as shown in Fig.(\ref{ab2}). And then insert   the row  $a$ into $\lambda^{''}$ and insert the row $c$ into $\lambda^{'}$ at the same time.
\begin{figure}[!ht]
  \begin{center}
    \includegraphics[width=3in]{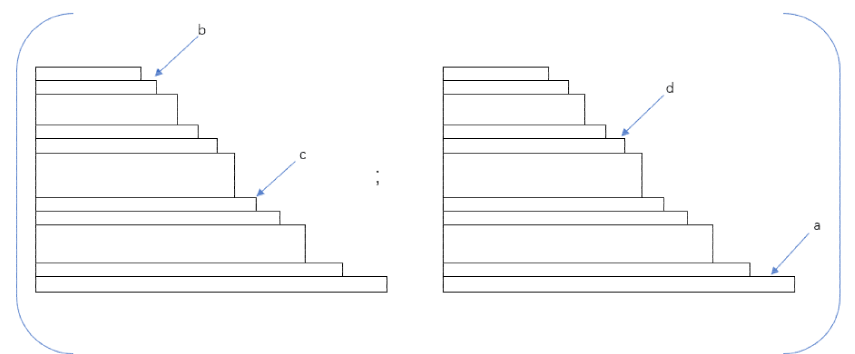}
  \end{center}
  \caption{Insertion of  the row $a$ of $\lambda^{'}$ into $\lambda^{''}$ and insertion of  the row $c$ of $\lambda^{''}$ into $\lambda^{'}$ at the same time.}
  \label{ab1}
\end{figure}
\begin{figure}[!ht]
  \begin{center}
    \includegraphics[width=3in]{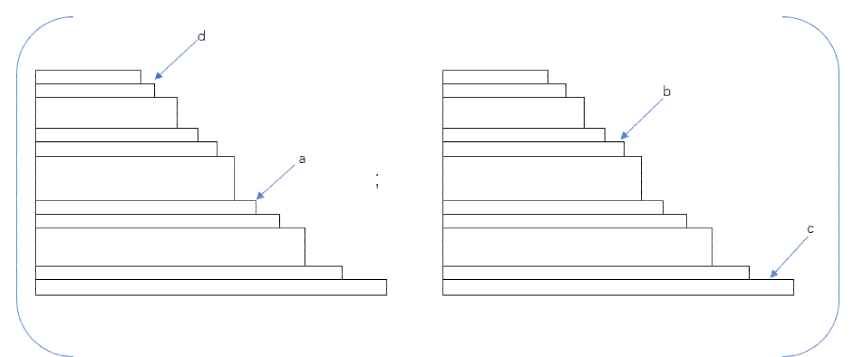}
  \end{center}
  \caption{Insertion of  the row  $b$ of $\lambda^{'}$ into $\lambda^{''}$ and insertion of  the row $d$ of $\lambda^{''}$ into $\lambda^{'}$ at the same time. }
  \label{ab2}
\end{figure}

We can discuss a more complicated case: $L(b)>L(a)>L(d)>L(c)$ in Fig.(\ref{abcd}).
To save space, we omit these proofs  which  are easy to understand.
The map $O$ in Fig.(\ref{o}) can be decomposed into to maps: move the longer two rows first, then move the shortest row. However, the move of the last row is the maps shown in Fig.(\ref{te}) and (\ref{to}).
In summary, the maps preserving symbol consist of the sequence  movements of  two rows. There are two  fundamental maps: $S$ and $D$ types.
\begin{flushleft}
\textbf{{$S$-type maps}}
\end{flushleft}

\begin{figure}[!ht]
  \begin{center}
    \includegraphics[width=2.5in]{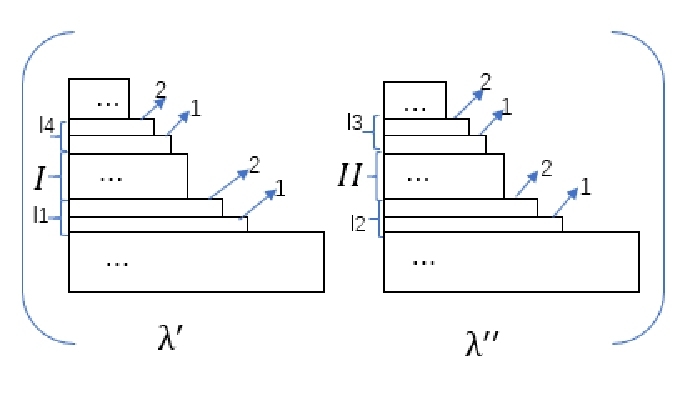}
  \end{center}
  \caption{Taking one row from $\lambda^{'}$  to  $\lambda^{''}$  and   taking one row from $\lambda^{''}$  to  $\lambda^{'}$. }
  \label{s}
\end{figure}
The  $S$ type  symbol preserving maps  which  take one row from $\lambda^{'}$  to  $\lambda^{''}$  and take one row from $\lambda^{''}$  to  $\lambda^{'}$. The most general model is  shown in Fig.(\ref{s}).
We can put the first or the second row of the pairwise rows in the  region $l1$ to     the positions $1$ or $2$  in the region $l2$. We can put the first or the second row of the pairwise rows in the region $l3$ to  the positions $1$ or $2$  in the region $l4$.

The rows in the regions $I$ and $II$ must have the opposite parities: $(odd, even)$ or  $(even, odd)$. Then the  number of the $S$ type maps  is
\begin{equation}\label{ns}
  NS=2^5.
\end{equation}

\begin{flushleft}
\textbf{{$D$-type maps}}
\end{flushleft}
The  $D$ type symbol preserving  maps which take two rows from $\lambda^{'}$  to  $\lambda^{''}$.
A  pair of even  rows can be obtained  from five  different blocks $\lambda^{'}$ as shown in Fig.(\ref{d1}).
This pairwise rows can be  inserted  into a block of $\lambda^{''}$, leading to five    different blocks as shown in Fig.(\ref{dd2}).

Considering  the  odd version of these processes,  the number of the $D$ type maps is
\begin{equation}\label{nd}
  ND=5*5*2.
\end{equation}
\begin{figure}[!ht]
  \begin{center}
    \includegraphics[width=4.4in]{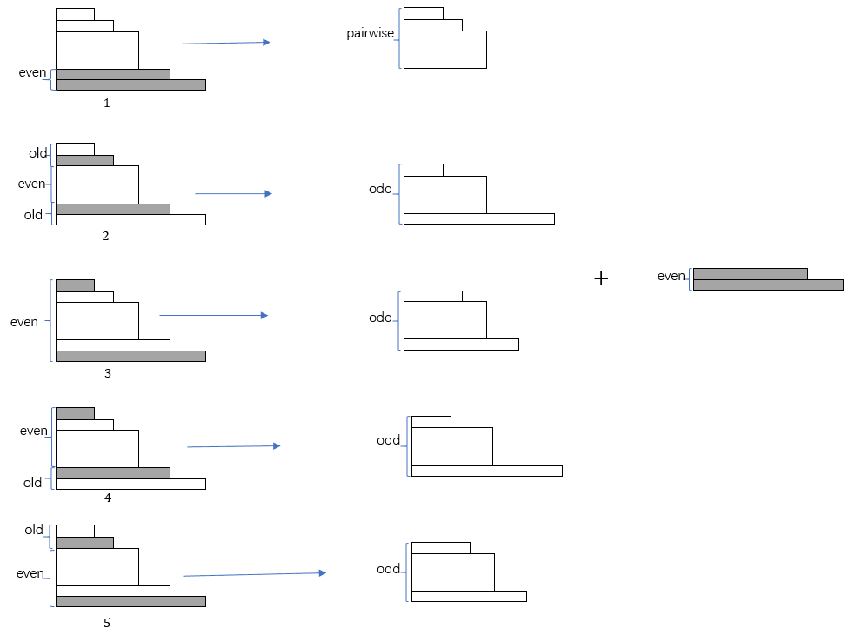}
  \end{center}
  \caption{Two rows of a block of $\lambda^{'}$ comprise an even  pairwise rows. }
  \label{d1}
\end{figure}
\begin{figure}[!ht]
  \begin{center}
    \includegraphics[width=4.4in]{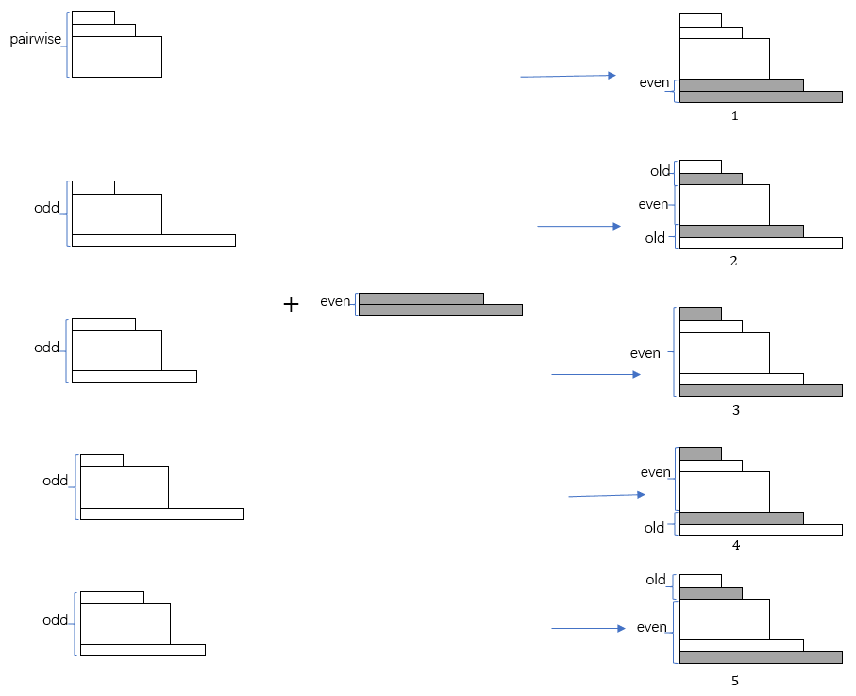}
  \end{center}
  \caption{An even  pairwise rows is inserted  into a block of $\lambda^{''}$. }
  \label{dd2}
\end{figure}

\begin{flushleft}
\textbf{{Summary}}
\end{flushleft}
All other symbol preserving maps can be constructed by $S$ and $D$ types of maps.
Then  to prove that symbol invariant implies fingerprint invariant is to   prove that these two fundamental maps  preserve the fingerprint invariant.
However, from formulas (\ref{ns}) and (\ref{nd}), the number of $S$ and $D$ types maps is
$$N=NS+ND=82,$$
which means that we should check these 82 cases to preserve the fingerprint invariant to prove the symbol invariant implies the fingerprint invariant. It is unrealistic.

In\cite{SW17-2}, the fingerprint maps are classified by decomposing  into several fundamental maps.
However, it is more complicated than that in Section \ref{mapf}, which is not conventional to prove  these fundamental maps  preserve the symbol invariant.

\newpage
\section{Rigid Surface Operators in $SO(11)$ and $Sp(10)$}
  The second  and the third columns   list   pairs of partitions corresponding to the surface operators in the $B_5$ and $C_5$ theories.
  The other  columns are  the  dimension,   symbol invariant,   and  fingerprint invariant of the rigid surface operators, respectively.
\begin{equation} {
\begin{array}{l@{\hspace{10pt}}l@{\hspace{10pt}}l@{\hspace{10pt}}l@{\hspace{10pt}}l@{\hspace{10pt}}c@{\hspace{10pt}}l}
\underline{Num}
&
\underline{Sp(10)}
&
\underline{SO(11)}
&
\underline{Dim}
&
\underline{Symbol}
&
\underline{Fingerprint}
\\[-0.5pt]
1
&
(1^{10}\,;\emp)
&
(1^{11};\emp)
 &
 0
 &
\left(\begin{array}{@{}c@{}c@{}c@{}c@{}c@{}c@{}c@{}c@{}c@{}c@{}c@{}c@{}c@{}c@{}}
 0&&0&&0&&0&&0&&0 \\
 &1&&1&&1&&1&&1&
\end{array}\right)
 &
[1^5;\emp]
 \\[-0.5pt]
 2
&
 (2\,1^{8}\,;\emp)
 &
(1;1^{10})
 &
 10
 &
\left(\begin{array}{@{}c@{}c@{}c@{}@{}c@{}c@{}c@{}c@{}c@{}c@{}c@{}c@{}}
 1&&1&&1&&1&&1 \\
&0&&0&&0&&0&
\end{array}\right)
 &
[1^4;1]
 \\[-0.5pt]
3
&
  (1^{8}\,;1^2)
 &
(2^2\,1^7;\emp)
 &
16
 &
\left(\begin{array}{@{}c@{}c@{}c@{}c@{}c@{}c@{}c@{}c@{}c@{}c@{}c@{}}
 0&&0&&0&&0&&0 \\
 &1&&1&&1&&2&
\end{array}\right)
 &
[2\,1^3;\emp]
 \\[-0.5pt]
 4
&
 (2^3\,1^4\,;\emp)
 &
(1^3;\,1^8)
 &
24
 &
\left(\begin{array}{@{}c@{}c@{}c@{}@{}c@{}c@{}c@{}c@{}c@{}c@{}}
 1&&1&&1&&1 \\
 &0&&0&&1&
\end{array}\right)
 &
[1^2;1^3]
\\[-0.5pt]
5
&
(2\,1^6\,;1^2)
 &
(1;2^2\,1^{6})
 &
 24
 &
\left(\begin{array}{@{}c@{}c@{}c@{}@{}c@{}c@{}c@{}c@{}c@{}c@{}c@{}c@{}}
 1&&1&&1&&1 \\
 &0&&0&&1&
\end{array}\right)
 &
[1^2;1^3]
 \\[-0.5pt]
 6
&
 (1^{6}\,;1^4)
 &
(2^4\,1^3;\emp)
 &
24
 &
\left(\begin{array}{@{}c@{}c@{}c@{}c@{}c@{}c@{}c@{}c@{}c@{}}
 0&&0&&0&&0 \\
 &1&&2&&2&
\end{array} \right)
 &
[2^2\,1;\emp]
 \\[-0.5pt]
7
&
 (2^4\,1^2\,;\emp)
 &
(1^7;1^4)
 &
28
 &
\left(\begin{array}{@{}c@{}c@{}c@{}c@{}c@{}c@{}c@{}c@{}c@{}}
 0&&0&&1&&1 \\
 &1&&1&&1&
\end{array}\right)
 &
[1;1^4]
 \\[-0.5pt]
 8
&
 (1^6;2\,1^2)
 &
(3\,2\,1^4;\emp)
 &
28
 &
\left(\begin{array}{@{}c@{}c@{}c@{}c@{}c@{}c@{}c@{}c@{}c@{}}
 0&&0&&1&&1 \\
 &1&&1&&1&
\end{array}\right)
 &
[1;1^4]
 \\[-0.5pt]
9
&
 (2\,1^4\,;1^4)
 &
(1^5;1^6)
 &
 30
 &
\left( \begin{array}{@{}c@{}c@{}c@{}c@{}c@{}c@{}c@{}}
 1&&1&&1 \\
 &1&&1&
\end{array} \right)
 & [\emp;1^5]
 \\[-0.5pt]
 10
&\,
 (2\,1^4\,;2\,1^2)
 &
(1;3\,2^2\,1^3)
 &
 34
 &
\left(\begin{array}{@{}c@{}c@{}c@{}@{}c@{}c@{}c@{}c@{}}
 1&&2&&2 \\
 &0&&0&
\end{array}\right)
 &
[21;2]
 \\[-0.5pt]
 11
&
  (1^2;2^3\,1^2)
 &
(2^2\,1;1^6)
 &
 34
 &
\left(\begin{array}{@{}c@{}c@{}c@{}@{}c@{}c@{}c@{}c@{}c@{}c@{}}
 1&&1&&1 \\
 &0&&2&
\end{array}\right)
 &
[3\,1;1]
 \\[-0.5pt]
 12
&
  (3^2\,2\,1^2;\emp)
 &
(1^3;2^2\,1^4)
 &
 34
 &
\left(\begin{array}{@{}c@{}c@{}c@{}c@{}c@{}c@{}c@{}}
 1&&1&&1 \\
 &0&&2&
\end{array}\right)
 &
[3\,1;1]
 \\[-0.5pt]
13
&
-
 &
 (2^2\,1^3;1^4)
&
36
 &
\left(\begin{array}{@{}c@{}c@{}c@{}@{}c@{}c@{}c@{}c@{}}
 0&&1&&1 \\
 &0&&2&
\end{array}\right)
 &
[3;1^2]
 \\[-0.5pt]
14
&
-
 &
(1^3;3\,2^21)
 &
45
 &
\left(\begin{array}{@{}c@{}c@{}c@{}c@{}c@{}}
2&&2 \\
 &1&
\end{array}\right)
 &
[3;2]
\end{array}
}\non
\end{equation}


\begin{thebibliography}{99}
\bibitem{CM93}
D.~H. Collingwood and W.~M. McGovern, { Nilpotent orbits in semisimple Lie
  algebras},
\newblock Van Nostrand Reinhold, 1993.

\bibitem{GW06}
S.~Gukov and E.~Witten, {Gauge theory, ramification, and the geometric
  {Langlands} program},  arXiv:hep-th/0612073

\bibitem{Wit07}
E.~Witten, {Surface operators in gauge theory},   {\it Fortsch. Phys.}, {\textbf
  55} (2007) 545--550.

\bibitem{GW08}
S.~Gukov and E.~Witten, {Rigid surface operators},   arXiv:0804.1561


\bibitem{Wy09}
N.Wyllard, { Rigid surface operators and $S$-duality: some proposals},   arXiv: 0901.1833



\bibitem{Lu79}
G.~Lusztig, { A class of irreducible representations of a Weyl group},  Indag.Math, 41(1979), 323-335.

\bibitem{Lu84}
G.~Lusztig, { Characters of reductive groups over a finite field},
\newblock Princeton, 1984.


\bibitem{Sp92}
N.~Spaltenstein, {Order relations on conjugacy classes and the
  {Kazhdan-Lusztig} map},  {\it Math. Ann.}, {\textbf 292} (1992) 281.

\bibitem{Montonen:1977}
C.~Montonen and D.~I. Olive, ``Magnetic monopoles as gauge particles?,'' {\em
  Phys. Lett.} {\bf B72} (1977)
117;

\bibitem{GNO76}
P.~Goddard, J.~Nuyts, and D.~I. Olive, {Gauge theories and magnetic charge},
  {\it Nucl. Phys.}, {\textbf B125} (1977)
1.

\bibitem{AKS06}
P.~C. Argyres, A.~Kapustin, and N.~Seiberg, {On {$S$-duality} for
  non-simply-laced gauge groups},  {\it JHEP}, {\textbf 06} (2006) 043,arXiv:hep-th/0603048


\bibitem{GM07}
  J.~Gomis and S.~Matsuura,
  {Bubbling surface operators and $S$-duality},  \\
  {\it JHEP}, {\textbf 06} (2007) 025,arXiv:0704.1657

\bibitem{DGM08}
N.~Drukker, J.~Gomis, and S.~Matsuura, {Probing $\cN=4$ SYM with surface
  operators},  {\it JHEP}, {\textbf 10} (2008) 048,  arXiv:0805.4199

\bibitem{GW14}
S.~Gukov, {Surfaces Operators},   arXiv:1412.7145


\bibitem{Shou-sc}
Shou, B.,  {   Symbol Invariant of Partition and Construction}, preprint, 31pp,  arXiv:1708.07084


\bibitem{SW17-2}
B.~Shou, and Q.~Wu, {Fingerprint Invariant of Partitions and Construction }, preprint, 23pp,  arXiv:1711.03443


\bibitem{HW07a}
M.~Henningson and N.~Wyllard, {Low-energy spectrum of {$\cN = 4$}
  super-{Yang-Mills} on {$T^3$}: flat connections, bound states at threshold,
  and {$S$-duality}},  {\it JHEP}, {\textbf 06} (2007), arXiv:hep-th/0703172

\bibitem{HW07b}
M.~Henningson and N.~Wyllard, {Bound states in {$\cN = 4$} {SYM} on {$T^3$}:
  {$\Spin(2n)$} and the exceptional groups},  {\it JHEP}, {\textbf 07} (2007) 084, arXiv:0706.2803


\bibitem{HW08}
M.~Henningson and N.~Wyllard, {Zero-energy states of $\cN = 4$ {SYM} on
  $T^3$: $S$-duality and the mapping class group},  {\it JHEP}, {\textbf 04} (2008)
  066, arXiv:0802.0660




\bibitem{rso}
C.Z. Li, B.Shou,  { Rigid Surface opetators and Symbol Invariant of Partitions}, Commun. Math. Phys. 405, 70 (2024), {\tt  arXiv: 1708.07388}.

\bibitem{Spaltenstein:1992}
N.~Spaltenstein, Order relations on conjugacy classes and the {Kazhdan-Lusztig} map, {\em Math. Ann.} {\bf 292} (1992) 281.


\bibitem{Lusztig:1984}
G.~Lusztig, { Characters of reductive groups over a finite field},\newblock Princeton, 1984.

\bibitem{Symbol 2}
G.~Lusztig, { A class of irreducible representations of a Weyl group}, {\em Indag.Math}, 41(1979), 323-335.

\bibitem{Sh11}
B.~Shou, J.F.~Wu and M.~Yu, { AGT conjecture and AFLT states: a complete construction}, preprint, 28 pp., arXiv:1107.4784




\bibitem{Localization}
N.~Nekrasov,  A.~Okounkov, {Seiberg-Witten theory and random partitions}, {\em The Unity of Mathematics},  525-596.  {\tt arXiv:hep-th/0306238}.




\end{thebibliography}
\end{document}